\documentclass{amsart}
\usepackage[T1]{fontenc}
\usepackage[utf8]{inputenc}
\usepackage{graphicx, subfigure}
\usepackage{amssymb,latexsym, color}
\usepackage{subfigure}
\usepackage{enumitem}
\usepackage{tikz}
\usetikzlibrary{decorations.markings}

\numberwithin{equation}{section}
\def\C{{\mathbb C}}

\def\N{{\mathbb N}}
\def\Z{{\mathbb Z}}

\def\R{{\mathbb R}}

\def\CP{{\mathbb C\mathbb P}}
\def\eps{{\epsilon}}

\newcommand{\Cal}{\mathcal}

\newcommand{\sminus}{\smallsetminus}
\newcommand{\interior}{\operatorname{int}}
\newcommand{\ind}{\operatorname{Ind}}
\def\res{\mathrm{Res}}

\usepackage{cleveref}
\theoremstyle{plain}
\newtheorem{lemma}{Lemma}[section]
\newtheorem{proposition}[lemma]{Proposition}

\newtheorem{theorem}[lemma]{Theorem}
\newtheorem{corollary}[lemma]{Corollary}

\theoremstyle{definition}
\newtheorem{definition}[lemma]{Definition}
\newtheorem{example}[lemma]{Example}
\newtheorem{remark}[lemma]{Remark}

\theoremstyle{remark}

\usepackage{xcolor}
\usepackage{xspace}

\title{Remarks on rational vector fields on $\CP^1$ }

\author[M. Klime\v{s}]{Martin Klime\v{s}}
\email{martin.klimes@univie.ac.at}

\author[C. Rousseau]{Christiane Rousseau}
\address{Christiane Rousseau, D\'epartement de
math\'ematiques et de statistique, Universit\'e de Montr\'eal, C.P. 6128,
Succursale Centre-ville, Montr\'eal (Qc), H3C 3J7, Canada.}
\email{rousseac@dms.umontreal.ca}

\thanks{The second author is supported by NSERC in Canada. }


\begin{document}
\tikzset{->-/.default=0.5, ->-/.style={decoration={markings, mark=at position #1 with {\arrow{>}}},postaction={decorate}}}

\date{\today}

\begin{abstract} In this paper we introduce geometric tools to study the families of rational vector fields of a given degree over $\CP^1$. To a generic vector field of such a parametric family we associate several geometric objects: a periodgon, a star domain and a translation surface. These objects generalize objects with the same name introduced in previous works on polynomial vector fields. They are used to describe the bifurcations inside the families. We specialize to the case of rational vector fields of degree $4$. 
\end{abstract}

\maketitle

\section{Introduction}
Polynomial and rational vector fields are important for many questions in holomorphic dynamics. For instance, the formal normal form of a parabolic point of a germ of holomorphic diffeomorphism of codimension $k$ (i.e. a multiple fixed point of multiplicity $k+1$) is given by the time-one map of a rational vector field $\dot z = \frac{z^{k+1}}{1+az^k}$. This is also true for the unfoldings: the formal normal form of a generic $k$-parameter parabolic point of a germ of holomorphic diffeomorphism of codimension $k$ is given by the time-one map of a family of rational vector fields \begin{equation}\dot z = \frac{P_\eps(z)}{1+a(\eps)z^k}= \frac{z^{k+1}+ \eps_{k-1} z^{k-1} + \dots +\eps_0}{1+a(\eps)z^k},\label{model_family}\end{equation}
and generalizations exist for non generic families and/or generic families depending on $m\neq k$ parameters. Similar theorems exist for resonant fixed points of a holomorphic diffeomorphism (when the multiplier at the fixed point is a root of unity of order $q$): there the normal form is the composition of a rotation of order $q$ with the time-one map of a rational vector field. 

For many of these problems, the fixed points are all in a small neighborhood of the origin while the poles are far from the origin. Hence the (real) phase portrait of the vector field  \eqref{model_family} (i.e. with real time) is close to the one of the polynomial vector field $\dot z = P_\eps(z)$. But there are new problems where this is not the case. For instance, when studying the unfolding of a resonant irregular singular point of a linear system, one of us had to deal with the system \eqref{model_family} with $k=1$ and $a$ that could be as large as $\frac1{\eps}$ \cite{Kl}.

Douady, Estrada and Sentenac \cite{DS} initiated geometric methods to study the real dynamics of complex polynomial vector fields, by introducing an invariant of analytic classification for generic polynomial vector fields. This invariant, later generalized for all polynomial vector fields (see for instance \cite{BD}), is composed of two parts: a topological (or combinatorial) part, and an analytic part. The set of structurally stable monic  polynomial vector fields is divided in $C_k= \frac{\binom{2k}{k}}{k+1}$ open sets separated by bifurcation surfaces where the Douady-Estrada-Sentenac invariant is discontinuous. 
Some attempts have been made to generalize this invariant to rational vector fields under the generic condition that they do not exhibit an annulus of periodic orbits \cite{Tom1}. But in general, the global dynamics of rational vector fields is not yet fully understood. 

In the paper \cite{CR}, a new tool was introduced to study polynomial vector fields, namely the \emph{periodgon}, or polygon of the periods. First introduced for the $1$-parameter family of vector fields $\dot z = z^{k+1}-\eps$, the tool was generalized in \cite{KR} to all polynomial vector fields, and it was analyzed in detail in the particular case of the $2$-parameter family of vector fields $\dot z = z^{k+1}+\eps_1z+\eps_0$. 
When all singular points of a polynomial vector field are simple  the periodgon is constructed as follows: First we associate to each simple singular point its \emph{periodic domain}, defined as the basin of the center at the singular point in the properly rotated vector field. 
The \emph{periodgon} is defined as the image in the time variable of the complement of the union of the periodic domains of all singular points.  It takes the form of a ``polygon'' on the translation surface of the time coordinate. 
The periodgon represents the ``core'' part of the dynamics from which the complete information can be read. Its construction has also its degenerate equivalent in case of vector fields with multiple singular points.
In a parametric setting the periodgon bifurcates when the boundary of some periodic domain is the union of several  homoclinic loops. 
In all the examples studied in \cite{CR}, \cite{KR} and \cite{R19}, the bifurcation diagram of the periodgon produces less open strata than that of the Douady-Estrada-Sentenac invariant. 

In this paper we generalize the construction of the periodgon to rational vector fields and we explore its role in describing the structure of the family of rational vector fields of degree $n$.  We also introduce a notion of \emph{translation model} of the vector field, as the surface $\CP^1\sminus\{\text{singularities of the vector field}\}$ equipped with a translation structure given by the rectifying coordinate for the vector field. This is a natural object that has been previously studied by Muciño-Raymundo and Valero-Vald\'es \cite{Mucino,Mucino-Valero,Mucino-Valero2}, and  it appears implicitly also in the work of Douady, Estrada and Sentenac \cite{DS}, and in that of Branner and Dias \cite{BD}.
The translation model is also obtained by regluing the periodgon and attaching semi-infinite cylinders to its sides in place of the periodic domains.

There are immediate differences between polynomial and rational vector fields. A first one is that rational vector fields have in general several poles and, generically,  the periodic domain of each singular point is attached to one of them. Hence we have configurations of periodic domains attached to the different poles, and we study these configurations. A second difference is that rational vector fields can have annuli of periodic solutions not included in the basin of  a center. 
Also, while the periodgon of a rational vector field that we define below will be generically unique, the chosen cuts are less natural than in the polynomial case since they are determined not only by the planar organization of the periodic domains, but also by their length. 
However, the associated translation model is unique and canonical, and it is the important invariant.

After introducing the definitions of periodgon and translation model, and stating some general theorems regarding rational vector fields of degree $d$, we study the case $d=4$. 

This leads to us ask the following general questions.
\bigskip

\noindent{\bf Question 1.} What are the possible configurations of periodic domains attached to the poles  in a rational vector field of degree $d$,  i.e. the partition of periodic domains between the different poles? What are the maximal and the minimal number of  periodic domains attached to a pole  depending on its order?  As a particular case,  can we have configurations with no periodic domain attached to a pole? 
\bigskip

\noindent{\bf Question 2.}  How many open sets in parameter space are necessary to cover the set of all rational vector fields of degree $d$ having a generic periodgon? For instance, the bifurcation creating an annulus of periodic solutions does not need a change of open set in parameter space to describe the periodgon. 

\bigskip \noindent{\bf Question 3.}  How many \lq\lq geometric types\rq\rq\ of periodgon can we have for generic rational vector fields of degree $d$? 
In this paper we will show that if  $d=4$, this number is $3$ (Theorem~\ref{thm:types}).

\bigskip \noindent{\bf Question 4.}  How many annuli of periodic solutions not surrounding a center are possible for a  rational vector field of degree $d$, and what are their relative positions if several are possible?

\section{General theory of rational vector fields on $\CP^1$}

We consider a rational vector field of the form 
$$\dot z = \frac{P(z)}{Q(z)}$$
on $\CP^1$. Since the Euler characteristic of $\CP^1$ is equal to $2$, it follows from the Poincar\'e-Hopf index theorem that the number of zeros minus the number of poles when counted with multiplicity is equal to $2$. Using a Moebius transformation we can suppose that $\infty$ is a regular point, in which case we can suppose that $\deg(P)-\deg(Q)=2$ and we say that the \emph{degree of the vector field} is the degree of $P$. 
When   $\deg(P)=2$, the vector field is polynomial and can be brought by an affine change of coordinate to $\dot z = z^2-\eps$. When $\deg(P)=3$, it is better to move the pole at infinity, thus transforming the vector field in a polynomial vector field. The study of cubic polynomial vector fields was done in \cite{R}. Hence the first truly rational case occurs when $\deg(P)= 4$. It will be studied below in Section~\ref{sec:deg4}.

\subsection{Singular points and poles} 
We consider a rational vector field $\dot z =V(z)= \frac{P(z)}{Q(z)}$ with $P$ and $Q$ relatively prime and $\deg(P)-\deg(Q)=2$. The following propositions are well known: see for instance \cite{BriTho, DS, GGJ1, Tom2}.
\\

\begin{proposition}~
\begin{enumerate}[wide=0pt, leftmargin=\parindent]
\item  The vector field is analytically linearizable in the neighborhood of any simple singular point $z_j$, which can only be:
\begin{itemize}
	\item a focus if $V'(z_j)\notin \R\cup i\R$;
	\item a center if $V'(z_j)\in i\R\setminus\{0\}$;
	\item a radial node if $V'(z_j)\in \R\setminus\{0\}$;
\end{itemize}

\item Multiple singular points are  parabolic points. The local analytic type of a parabolic singular point $z_j$ of order $k+1$ (codimension $k$) is completely determined by its order and the its period (or dynamical residue) $\nu_j=2\pi i\,\res_{z=z_j}\frac{dz}{V(z)}$; 
in particular, it is locally biholomorphically equivalent to 
$\dot z= \frac{(z-z_j)^{k+1}}{1+\frac{\nu_j}{2\pi i}(z-z_j)^k}$.
It has $2k$ sepal zones, separated alternatingly by $k$ \emph{attractive} and $k$ \emph{repulsive directions} 
given by $\Im\left((z-z_j)^k\right)=0$
(see Figure~\ref{parabolic-pole}(a)). 

\item The poles are given by the zeros of $Q(z)$. In the neighborhood of a pole $p_j$ of order $k-1$, the vector field is analytically conjugate to $\dot z = \frac{1}{(z-p_j)^{k-1}}= \frac{(\bar z-p_j)^{k-1}}{|z-p_j|^{2k-2}}$. 
\end{enumerate}
\end{proposition}

\begin{figure}\begin{center}
\subfigure[A parabolic point of order $4$]{\includegraphics[width=4.5cm]{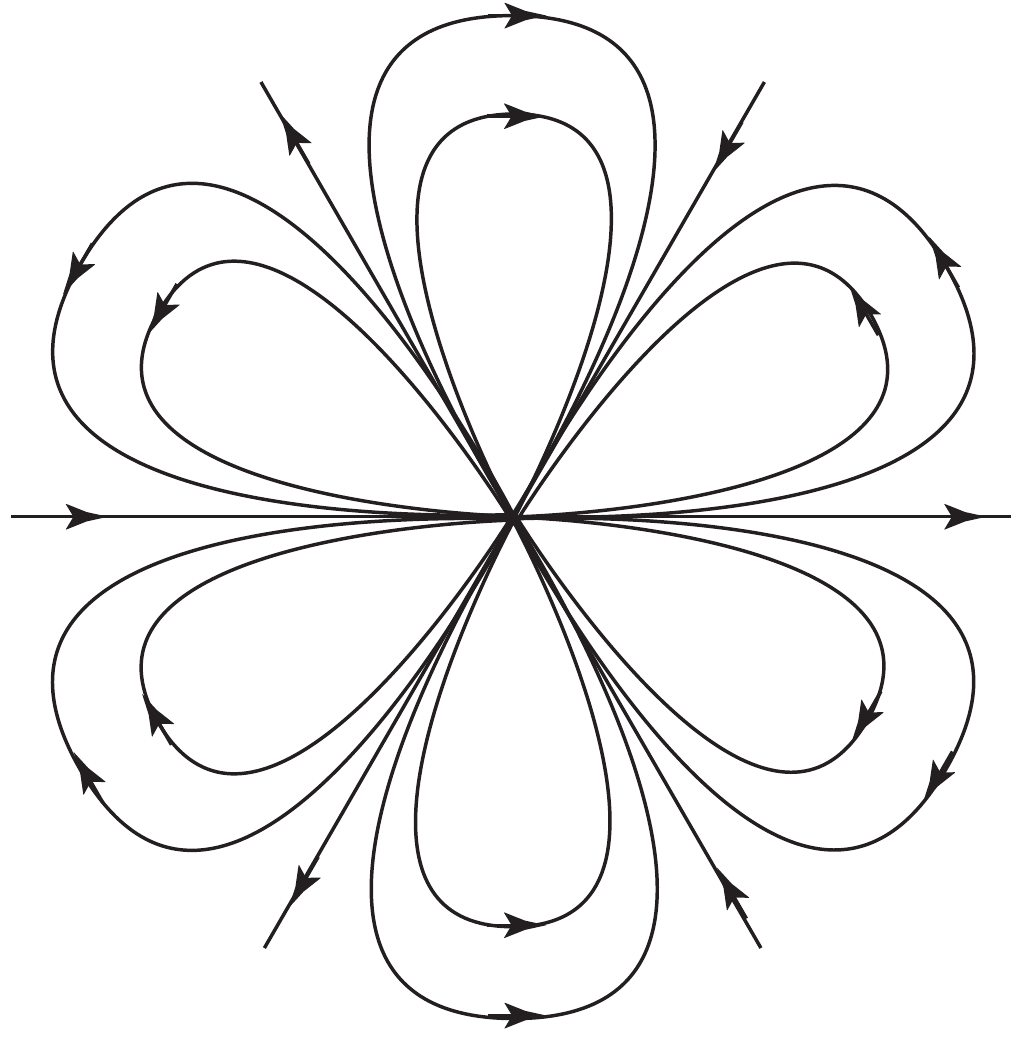}}\qquad\qquad 
\subfigure[A pole of order $2$]{\includegraphics[width=5cm]{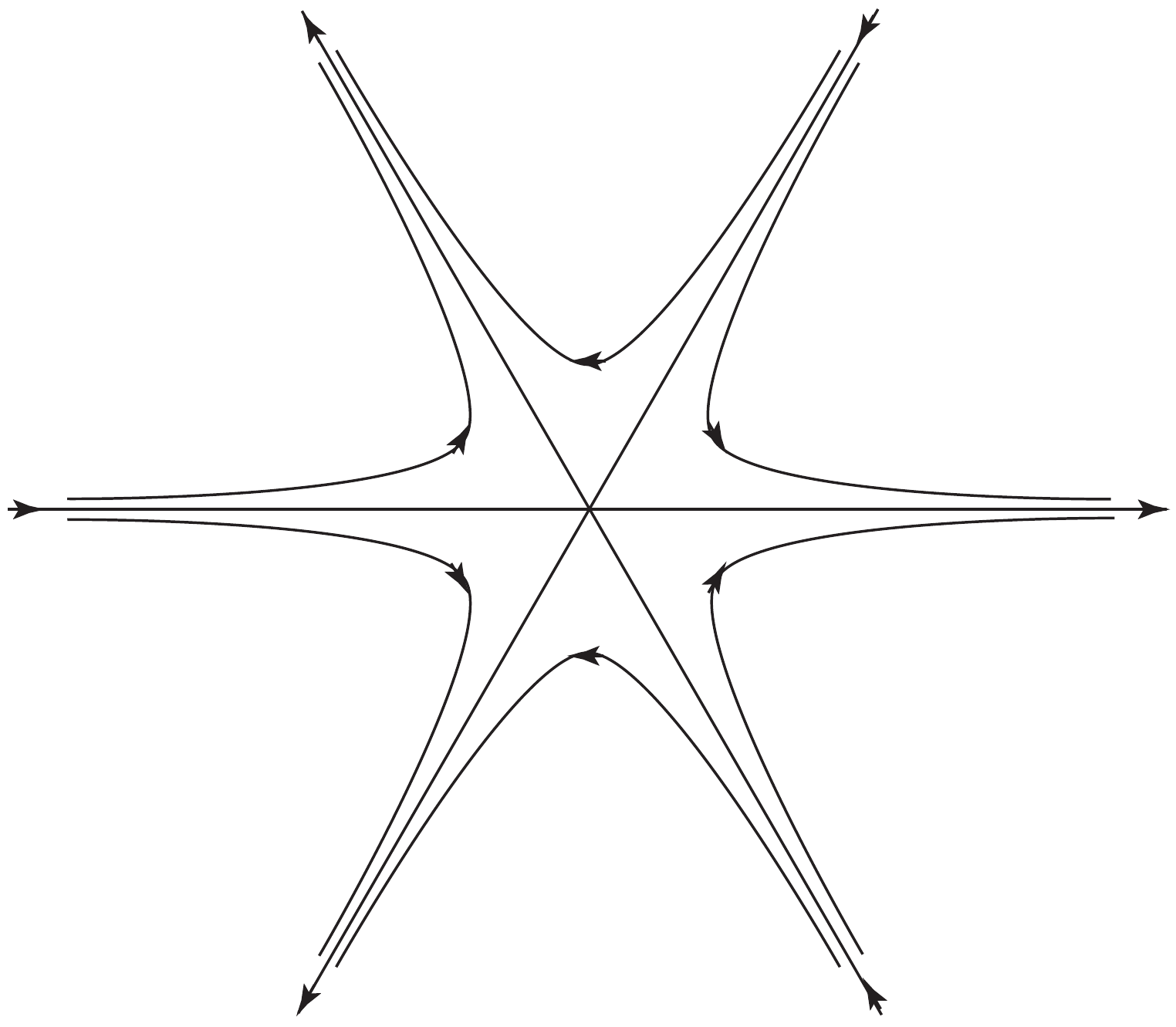}}
\caption{The phase portraits in the neighborhood of a parabolic point and in the neighborhood of a pole.}\label{parabolic-pole}\end{center}\end{figure}

\subsection{Global organization of trajectories} 
The following proposition is well known (for instance \cite{Ben},\cite[Theorem 2]{Ha66}, \cite[Theorem 9.4]{Stre}).\\

\begin{proposition}
\begin{enumerate}[wide=0pt, leftmargin=\parindent]
\item Every periodic trajectory of a rational vector field on $\CP^1$ belongs to a maximal open domain consisting of periodic trajectories (all of the same period)
and possibly of a center equilibrium point, the boundary of which is formed by a union of homoclinic or heteroclinic separatrices. 
\item  All non-periodic trajectories start and terminate at either a singular point or a pole. 
\end{enumerate}
\end{proposition}

\begin{proposition}\label{configurations} \cite{Tom2} 
\begin{enumerate}[wide=0pt, leftmargin=\parindent]
\item All separatrices of poles either land at a singular point or merge with another separatrix in a homoclinic or heteroclinic connection. 
\item At least one separatrix of some pole lands at each singular point of focus or node type, and at least one separatrix of some pole lands at a parabolic point tangent to each attractive/repulsive direction.  \end{enumerate}
\end{proposition}

The following lemma (see for exemple \cite[Theorem 1]{Ha66}) is a simple consequence of Poincar\'e-Hopf index theorem. 

\begin{lemma}\label{lemma:configurations}  
Each periodic orbit of a rational vector field surrounds $n+1$ zeros and $n$ poles for some $n\in \Z_{\geq0}$ (multiplicities taken into account).
\end{lemma}
\begin{proof} Simple zeros of a rational vector field have index $+1$ and simple poles have index $-1$, since the trajectories are organized as for a saddle point (the line field is the same as that of the vector field $\dot z = C\bar z$). 
It is known that the sum of the indices of the \lq\lq singularities\rq\rq\ inside a closed trajectory is equal to $+1$. \end{proof}

\begin{corollary}
A periodic trajectory of a polynomial vector field is inside the basin of a center. 
\end{corollary}
\begin{proof}
If a periodic trajectory or a polynomial vector field of degree $d$ surrounds more than one singular point, then it must surround a pole. But the only pole is at infinity and has degree $d-2$. Hence the periodic trajectory must surround $d-1$ singular points. Since it divides $\CP^1$ in two components, it is in the basis of the $d$-th singular point which is  a center. 
\end{proof}

\begin{definition}
	The closure in $\CP^1$ of the union of all separatrices of all poles is called the \emph{separatrix graph}. The connected components of its complement are called \emph{zones}. 	
\end{definition}

The topological organization of the real phase portrait of the vector field is completely determined by the separatrix graph.
There are four kinds of zones that can occur in a rational vector field on $\CP^1$:
\begin{enumerate}
	\item \emph{$\alpha\omega$-zone:} all trajectories have the same $\alpha$-limit and the same $\omega$-limit  which are two different equilibria (simple or multiple)
	\item \emph{odd/even sepal zone:} all trajectories have the same $\alpha$-limit and $\omega$-limit  which is  a parabolic equilibrium,
	\item \emph{odd/even center zone} consisting of periodic counter-clockwise/clockwise periodic orbits of the same period 
	$\nu_j>0$ around a center equilibrium point $z_j$ (which is included in the zone),
	\item \emph{periodic annulus} consisting of periodic orbits of the same period $\nu>0$  and surrounding on each side (on $\CP^1$) at least one pole.
\end{enumerate}

\subsection{The translation model of a rational vector field}

In this section we will move to the rectifying coordinate $t$ defined by 
\begin{equation}t = t(z) = \int\frac{Q(z)}{P(z)}\;dz.\label{eq:t}\end{equation}
The Riemann surface of $t(z)$ is a translation surface with the coordinate $t$ unique up to a translation. The periodgon and the star domain that we define below lie on this surface. 
The quotient of this surface by the deck transformations of the projection $t(z)\to z$, equipped with the form $dt$,  will be called the \emph{translation model} for the vector field.

\begin{definition}
We define the \emph{translation model} for the vector field $\dot z = \frac{P(z)}{Q(z)}$ as a triple
$(\Cal S,\Sigma, dt)$ where
\begin{enumerate}
	\item $\Cal S$ is the smooth manifold $\CP^1\sminus\{z:P(z)=0\}$,
	\item $\Sigma\subset\Cal S$ is the finite set of poles $\{z:Q(z)=0\}$,
	\item $\Cal S\sminus\Sigma$ is a translation surface, with local translation atlas given by the restrictions of the map \eqref{eq:t} to simply connected domains,
	\item the neighborhood of each point of $\Sigma$ corresponding to a pole of order $m-1$ is a cone of angle $2m\pi$,
	\item $dt$ is a translation invariant abelian form on $\Cal S\sminus\Sigma$. 
\end{enumerate}
In another words, the translation model is obtained by cutting $\CP^1\sminus\{z:P(z)=0\}$ into simply connected pieces,
taking their image by $t(z)$ and gluing them back together to create an abstract translation surface on which the real dynamics of the vector field is represented by the horizontal flow of $\dot t=1$.
\end{definition}

The translation model is a very natural object that was previously considered by other authors; in particular it appears in \cite{Mucino-Valero, Mucino-Valero2, Mucino}. 

\begin{proposition}
	Two rational vector fields are globally conjugated (by a Moebius transformation) if and only if their translation models are isomorphic, i.e. if there exists a biholomorphism of the translation surfaces sending each conic singularity  to a conic singularity with the same  angle and preserving the form $dt$.
\end{proposition}

\begin{proof}

There is a bijection $h_l$ between $\CP^1\sminus\{z:P_l(z)\}$ and the domain of the translation model of $\dot z = \frac{P_l(z)}{Q_l(z)}$, which is biholomorphic outside of the poles and sends the poles to the  conic singularities. Let $H$ be an isomorphism between the translation surfaces. Note that $H$ sends the conic singularities to conic singularities of the same type. Also it preserves the vector field $\dot t =1$. Then $h_2^{-1}\circ H\circ h_1$ is a global conjugacy of the two vector fields on $\CP^1\sminus\{z:P_l(z)=0\}$ that extends to the punctures (because it is bounded there), and that is analytic at the poles;
hence it is a Moebius transformation.
\end{proof}

\subsection{Periodic, parabolic and annular domains}

\begin{definition}
To each equilibrium point $z_j$ one associates its \emph{period} (also called \emph{dynamical residue}) as the ``travel time'' along a simple loop around the point:
$$\nu_j=2\pi i \;{\rm Res}_{z_j} \frac{Q(z)}{P(z)}.$$
By the residue theorem $\nu_1+\dots + \nu_d=0.$
\end{definition}

If $z_j$ is a simple equilibrium point, then its period is $\nu_j= 2\pi i\,\frac{Q(z_j)}{P'(z_j)}$, and
$z_j$ is a center of the rotated vector field 
\begin{equation}\dot z = e^{i\arg \nu_j}\frac{P(z)}{Q(z)}. \label{rotated_vf}\end{equation}

\begin{definition}~ 
	\begin{enumerate}[wide=0pt, leftmargin=\parindent]
		\item For a simple equilibrium point $z_j$, the \emph{periodic domain} of $z_j$  is the periodic basin of the center (also called center zone) at $z_j$ of \eqref{rotated_vf}. 
		The boundary of the periodic domain of $z_j$ consists of one or several homoclinic or heteroclinic connections of \eqref{rotated_vf}. Generically it is a single homoclinic loop through one pole, which we then call the \emph{homoclinic loop} of $z_j$. 
		\item  For a multiple equilibrium point $z_j$, the \emph{parabolic domain} of $z_j$ is the union of all the sepal zones of $z_j$ in \eqref{rotated_beta} for all $\beta\in\R$  (see Figure~\ref{parabolic_domains}).
		\item An \emph{annular domain} is a periodic annulus of \eqref{rotated_beta} for some $\beta\in\R$.
		\item An \emph{end} of a periodic/parabolic/annular domain refers to a ``corner'' of the domain at an adjacent pole. The domain may have several ends at the same pole. 
	\end{enumerate}
\end{definition}

\begin{figure}\begin{center}
\subfigure[]{\includegraphics[width=4cm]{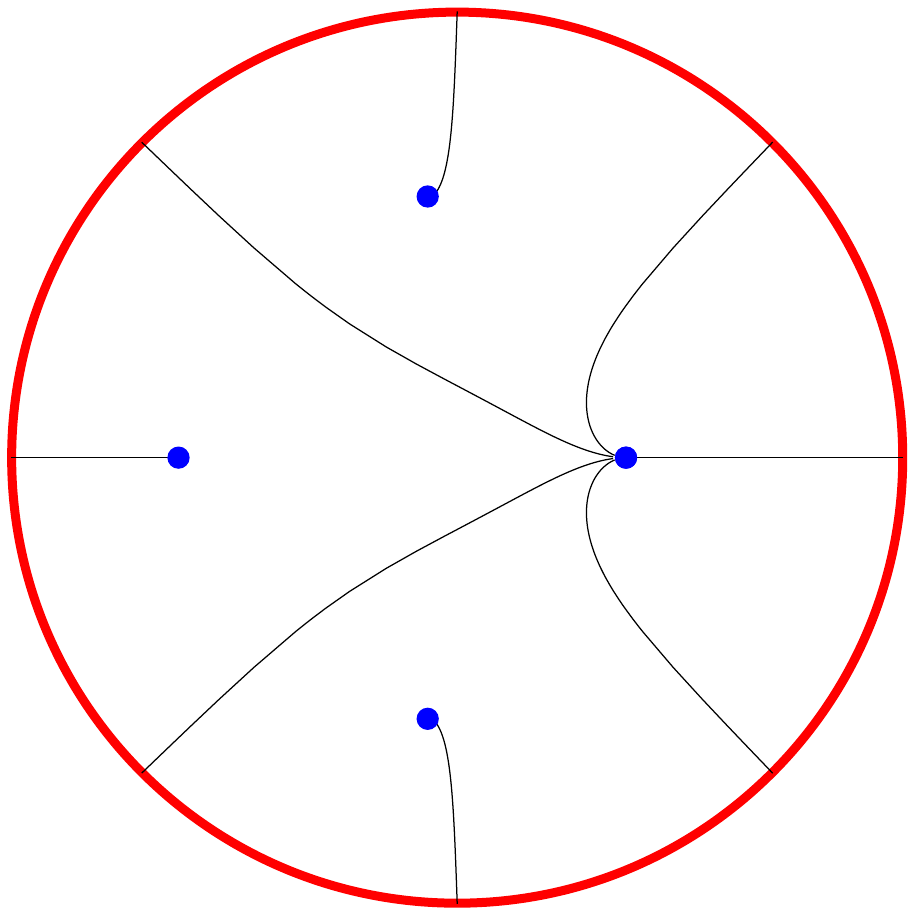}}\qquad \subfigure[]{\includegraphics[width=4cm]{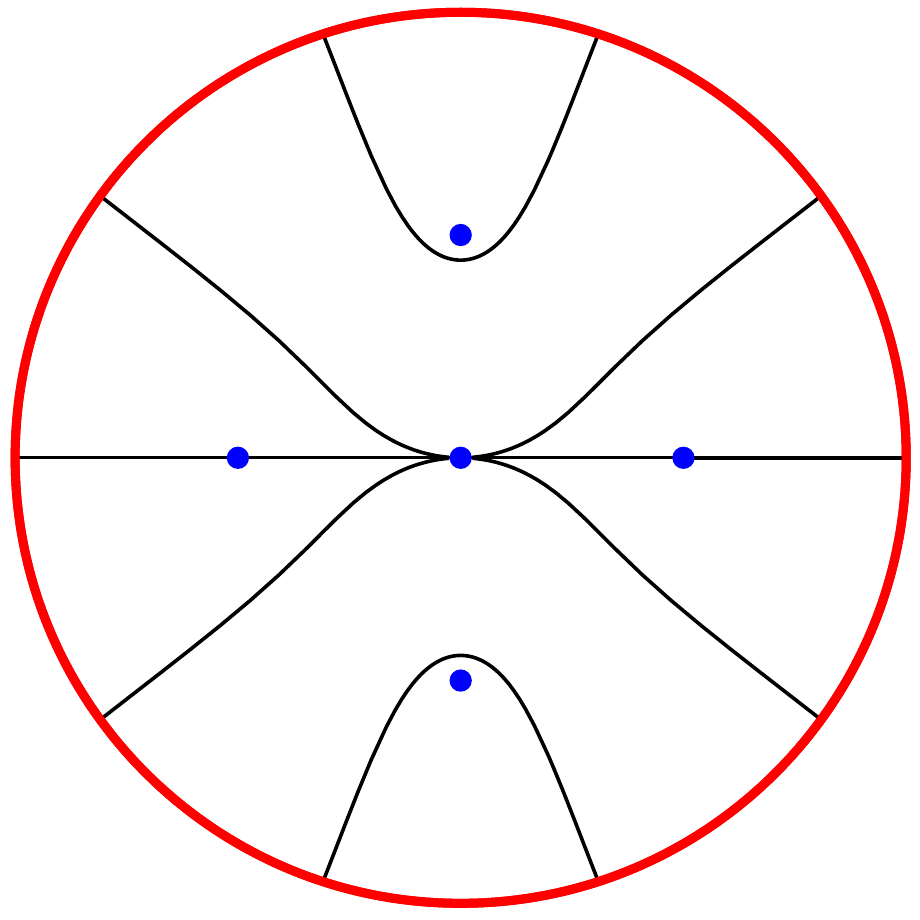}}\\ \subfigure[]{\includegraphics[width=4cm]{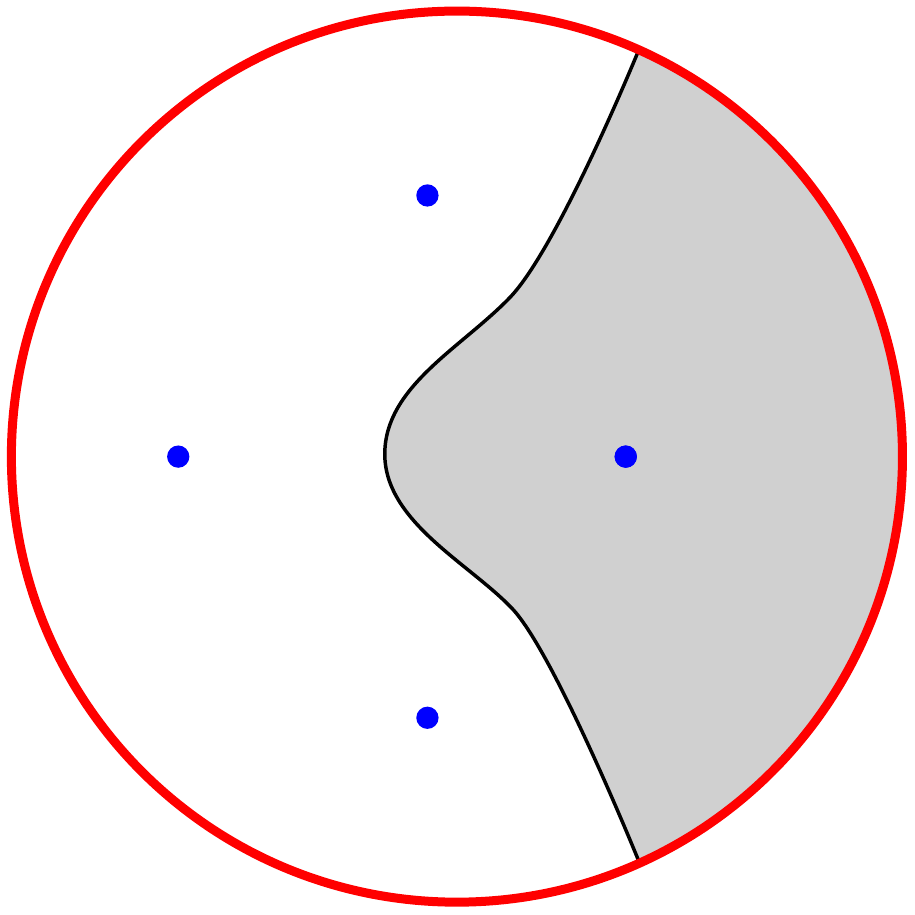}}\qquad \subfigure[]{\includegraphics[width=4cm]{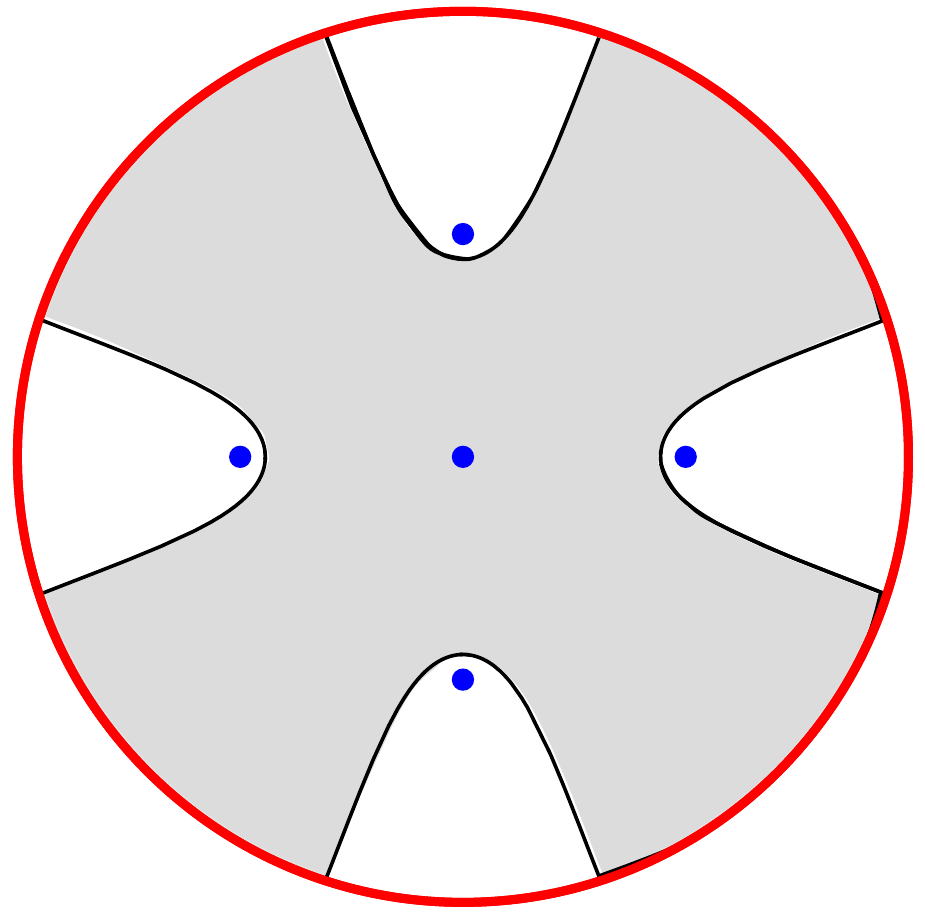}}
\caption{In (a) and (b) the  vector fields $\dot z = z^5-5z+4$ (see \cite{KR}) and  $\dot z= z^6-z^2$ (see \cite{R19}), and in (c) and (d) their respective parabolic domains at $z=1$ and $z=0$.}\label{parabolic_domains}\end{center}\end{figure}

\begin{definition} Given a rational vector field, $\dot z =\frac{P(z)}{Q(z)}$, we consider the associated family of rotated vector fields 
\begin{equation}\label{rotated_beta}
\dot z =e^{i\beta}\frac{P(z)}{Q(z)},\qquad \beta\in \R.
\end{equation}
Any homoclinic or heteroclinic connection appearing in the rotating family for some $\beta\in\R$ is called a
\emph{saddle connection} of the rational vector field. Note that it is oriented. \end{definition}

\begin{remark} We will consider simultaneously saddle connections corresponding to different angles of rotation $\beta$. \end{remark}

\begin{lemma}
	The parabolic domain of a multiple equilibrium $z_j$ covers a full neighborhood of $z_j$.
	Its boundary consists of a finite number of saddle connections.
\end{lemma}

\begin{proof}
	Locally, the sepal zones of the equilibrium $z_j$ rotate with $\beta$, hence their union covers a full neighborhood of $z_j$.
	For each $\beta$ the sepal zones are bounded by separatrices. A separatrix landing at the equilibrium will be completely contained in a sepal zone for any $\tilde\beta$ either on the left side or on the right side of $\beta$, and sufficiently close to $\beta$. On the other hand, a homoclinic separatrix bordering a sepal zone of $z_j$ for some $\beta$ will be also bordering the parabolic domain $z_j$, as follows from Proposition~\ref{prop:periodicdomains} below.
	There is only a finite number of ends to the parabolic domain (there is only a finite number of poles, each can host only finitely many ends), hence the boundary consists of only finitely many saddle connections.
\end{proof}	

\smallskip

\begin{proposition}~ \label{prop:periodicdomains}
	\begin{enumerate}[wide=0pt, leftmargin=\parindent]
		\item The periodic and parabolic domains of all the equilibria act as trapping regions for the whole rotating family of vector fields 	\eqref{rotated_beta}: no trajectory of \eqref{rotated_beta} for any $\beta$ can both enter and leave a periodic or parabolic domain.
		In particular, no saddle connection can intersect the periodic or parabolic domain of an equilibrium. 
		\item The periodic, parabolic and annular domains are all disjoint. 
		\item If the degree is $>2$ and  if all singular points are simple and the boundary of each periodic domain consists of exactly one homoclinic loop, then these homoclinic loops are disjoint too.
	\end{enumerate}
\end{proposition}

\begin{proof}
	The proof is the same as that of \cite{KR} in the polynomial case and uses well-known facts about families of rotated vector fields (see for instance \cite{Duf} or \cite{P}), namely that trajectories of the family of rotated vector fields \eqref{rotated_beta} for different values of $\beta\in \R$ can have at most one intersection point. 
	Also, separatrices of a pole for different $\beta$'s can only intersect at the pole.
\end{proof}

The following proposition is immediate.

\begin{proposition}\label{prop:translationmodel}
Let us now consider the images of the periodic, parabolic and annular domains by the map \eqref{eq:t}.
\begin{enumerate}[wide=0pt, leftmargin=\parindent]
	\item The image of an annular domain is a cylinder of finite area in the translation model, the boundary of which consists of two  connected components  
	formed by saddle connections  with the same angle $\beta$.
	In particular, each of the two boundary components contains at least one pole (element of $\Sigma$).
	\item The image of a periodic domain minus the singular point is a semi-infinite cylinder (a ``hose'' of flat pants) in the translation model, the boundary of which  has a single component formed by saddle connections   with the same angle $\beta=\arg\nu_j$.  
	\item The image of the parabolic domain of an equilibrium of multiplicity $k+1$ minus the singular point itself is a union of $2k$ half-planes, 
	corresponding to the sepal zones for any fixed $\beta$, and of a finite number of half-strips, corresponding to inter-sepal regions 
	each attached to a single component of the boundary of the domain  (see Figure~\ref{fig:par_domains}). 
	In total, it covers a punctured cone of angle $2k\pi$ at infinity (in a projective coordinate). \end{enumerate}\end{proposition}	

\begin{figure}\begin{center}
\subfigure[]{\includegraphics[width=4cm]{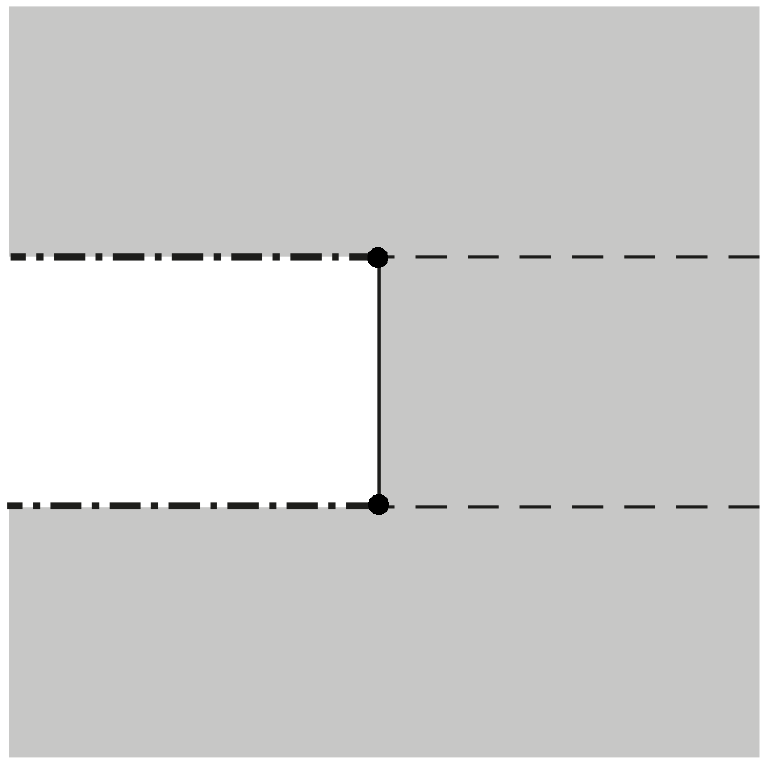}}\qquad\qquad\qquad\subfigure[]{\includegraphics[width=4cm]{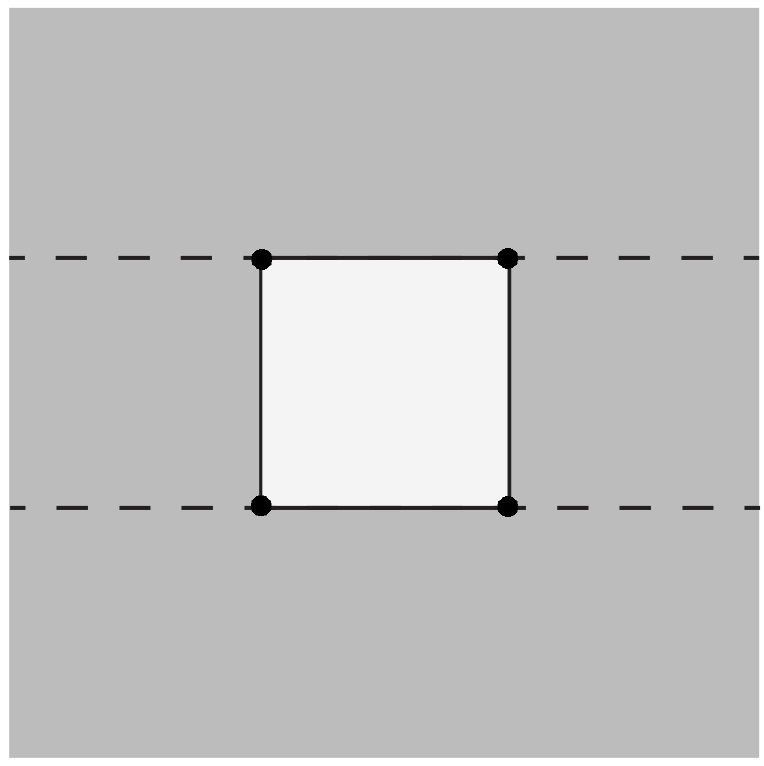}}\\
\subfigure[]{\includegraphics[width=4cm]{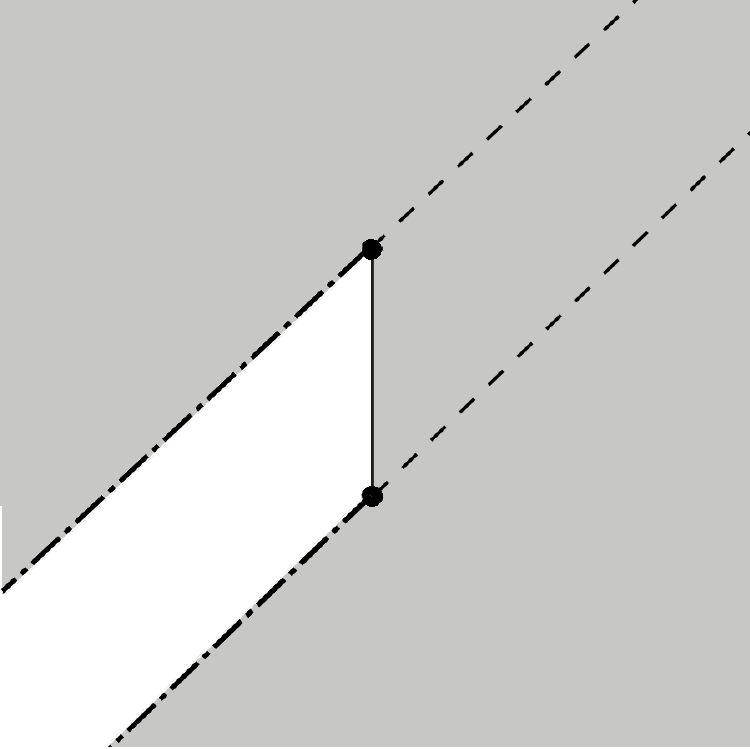}}\qquad\qquad\qquad\subfigure[]{\includegraphics[width=4cm]{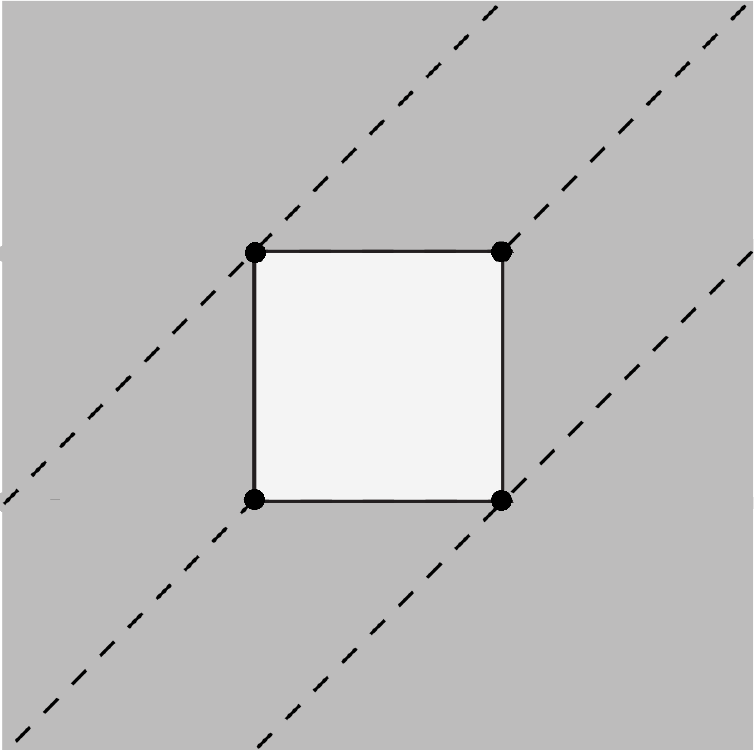}}
\caption{The images in $t$-space of the respective parabolic domains of the  vector fields  $\dot z = z^5-5z+4$ in (a), (c), and  $\dot z= z^6-z^2$ in (b), (d) of Figure~\ref{parabolic_domains} in the translation model. In (a), resp. (c), we glue together the left horizontal, resp. slanted, cuts. The thin dashed lines show the decomposition into sepal and inter-sepal regions for the angle $\beta=0$ in (a), (b), and $\beta=\frac{\pi}{4}$ in (c), (d).}\label{fig:par_domains}
\end{center}\end{figure}

\begin{proposition}[Muciño-Raymundo, Valero \cite{Mucino-Valero}]~
	\begin{enumerate}[wide=0pt, leftmargin=\parindent]
		\item 	 Each annular domain contains infinitely many saddle connections. More precisely, for each pair of ends on opposite boundaries of the domain, there are $\Z$-many different saddle connections (see Example~\ref{example1} and Figure~\ref{fig:example1}): if the $0$-th such connection is selected as (one of) the shortest, then the 
		$n$-th such connection is the one for which the closed curve obtained by composing it with the reverse of the $0$-th one
		has a turning number $n$.
		As $n\to\pm\infty$ the angle $\beta_n$ in \eqref{rotated_beta} of the $n$-the saddle connection tends to the angle $\beta$ of the annular domain.
		\item There is only a finite number of saddle connections that are not completely contained in one of the annular domains.
	\end{enumerate}
\end{proposition}

\begin{example}\label{example1} We consider the system 
\begin{equation}\dot z = i\frac{z^4+1}{1-z^2}.\label{vf_example}\end{equation} It has four centers at $z_j= e^{i\frac{\pi(1+2j)}4}$, $j=0, \dots, 3$, with periods $\pm\frac{\pi}{ \sqrt{2}}$,  and two poles at $z=\pm1$. Moreover, the  imaginary axis is invariant, and hence a periodic orbit on $\CP^1$. It belongs to a family of periodic orbits. 
The phase portrait appears in Figure~\ref{fig:example1}. There is a bi-infinite sequence of $\beta_n$ such that the corresponding rotated vector fields \eqref{rotated_beta}  has a saddle connection.\end{example}
\begin{figure} \begin{center}
\includegraphics[height= 4.5cm]{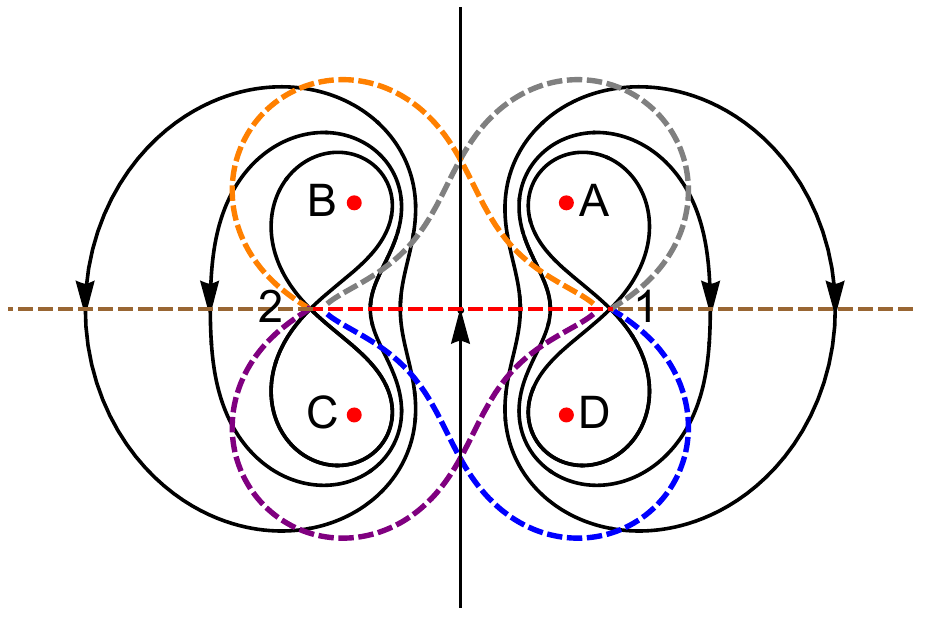}
\caption{The phase portrait of system \eqref{vf_example}. There is an annulus of periodic solutions on the sphere $\CP^1$, and a bi-infinite sequence of $\beta_n$ such that the corresponding rotated vector fields \eqref{rotated_beta}  has a saddle connection.}\label{fig:example1}\end{center} \end{figure}

\subsection{Chains of saddle connections}

The following proposition is a form of Poincar\'e-Hopf lemma for regions bounded by chains of saddle connections. 

\begin{proposition}\label{prop:index2} 
	 Let $\Gamma\subset\CP^1$ be a (union of) positively oriented closed curve(s) consisting  each of a finite number of a saddle connections forming  the boundary of an open set $\interior(\Gamma)\subset\CP^1$. 	Then
	\begin{equation}\label{formula}
	\#_{sc} \Gamma-\tfrac{1}{\pi}\sum_{\Gamma\cap\Sigma}\sphericalangle=2\chi(\interior(\Gamma))-2\sum_{z\in\interior(\Gamma)}\ind_zX,
	\end{equation} 
	where
	\begin{itemize}
		\item $\#_{sc} \Gamma$ is total the number of the oriented saddle connections appearing in all components of $\Gamma$ (the same  curve  may appear twice as two saddle connections with opposite orientations),
				\item $\sum_{\Gamma\cap\Sigma}\sphericalangle$ is the sum of the angles of $\interior(\Gamma)$ at its ends (at the points of $\Gamma\cap\Sigma$) measured in the translation model (i.e. each angle is the angle between the saddle connections at the pole  on $\CP^1$multiplied by $m+1$, where $m$ is the  multiplicity of the pole),
		\item $\chi$ is the Euler characteristic,
		\item $\ind_zX$ is the Poincar\'e-Hopf index of the vector field $X=\frac{P(z)}{Q(z)}\frac{\partial}{\partial z}$ at a point $z$, that is $m$, if $z$ is a singularity of multiplicity $m$, and $-m$, if $z$ is a pole of multiplicity $m$.
	\end{itemize}
\end{proposition}

\begin{proof}
We shall show that the formula is stable by the operations 
\begin{enumerate}[wide=0pt, leftmargin=\parindent]
	\item adding new points to $\Sigma$ as poles of multiplicity 0;
	\item cutting the set $\interior(\Gamma)$ by additional saddle connections (which then appears twice in the new $\Gamma$ with two opposite orientations).
\end{enumerate}
By this procedure, the domain can be cut into a union of nonsingular triangles and of periodic and parabolic domains. 
Indeed, with every singularity in $\interior(\Gamma)$, then $\interior(\Gamma)$ contains also its periodic/parabolic domain relative to $\Sigma$, since no saddle connection can cut this domain by Proposition~\ref{prop:periodicdomains}. Therefore it will be enough to show that the formula \eqref{formula} is valid for these domains. We do this in (3).

\begin{enumerate}[wide=0pt, leftmargin=\parindent]
\item Adding a new point to $\Sigma$:
\begin{enumerate}
	\item a point of $\interior(\Gamma)$: it increases the Poincar\'e-Hopf index by 0;
	\item a point of $\Gamma$: it divides a saddle connection into two pieces and adds an interior angle $\pi$. 
\end{enumerate}
\item Cutting $\interior(\Gamma)$ along a new saddle connection increases the number of saddle connections by 2 while
\begin{enumerate}
	\item if both endpoints are on $\Gamma$: either it increases the number of connected components of $\interior(\Gamma)$ by 1 or it decreases the number of holes in $\interior(\Gamma)$, in both cases increasing the Euler characteristic by 1;
	\item if one endpoint is on $\Gamma$ and the other in $\interior(\Gamma)$: an interior pole of order $m-1$ stops being counted by the double of its Poincar\'e-Hopf index $m-1$ and becomes counted by an angle $2m\pi$;
	\item if both endpoints are in $\interior(\Gamma)$: this creates a new hole, thus decreasing the Euler characteristic by 1, while also changing how the endpoints are counted.
\end{enumerate}
\item The formula \eqref{formula} is clear for simply connected domains with no singularities, as well as for periodic domains.
Let us show it for parabolic domains. By Proposition~\ref{prop:translationmodel} for a generic angle $\beta$ the parabolic domain decomposes into a union of $2m$ single-ended sepal zones (corresponding to half-planes with a single point of $\Sigma$ on the boundary line), each contributing by an angle $\pi$, and of a number of inter-sepal regions (corresponding to half-strips with two corner points of $\Sigma$), one for each  
saddle connection on the boundary of the domain, where the total angle $\pi$ of the half-strip cancels the contribution of the saddle connection.\end{enumerate}\end{proof}

\begin{remark}
One can always add any regular point to the set of poles $\Sigma$ as a pole of order zero, thus enlarging the set of saddle connections. In particular, Lemma~\ref{lemma:configurations} becomes a special case of Proposition~\ref{prop:index2} if one adds to $\Sigma$ a point on the periodic trajectory.
\end{remark}

\begin{proposition}\label{prop:fundamentalgroupoid}
	The fundamental groupoid $$\Pi_1(\Cal S,\Sigma)=\Pi_1(\CP^1\sminus\{\text{singularities}\},\ \{\text{poles}\})$$ is generated by the set of saddle connections. Different saddle connections are not end-point-homotopic.
\end{proposition}

\begin{proof}
	If two saddle connections are end-point-homotopic, then by  Proposition~\ref{prop:index2} the sum of the angles between them at their end-points is null, meaning that they are equal.
	
	The boundary of each periodic/parabolic domain is a chain of saddle connections forming a simple loop around the singularity with a base-point at a pole. All we need to show is that each two poles are connected by a chain of saddle connections.
	
	To each pole we associate the open domain swept by its separatrices in the rotating family \eqref{rotated_beta}, i.e. all the points on the translation model connected to the pole by a straight segment.
	If non-empty, the boundary of this domain in $\Cal S$  must contain another pole connected to the original pole by a saddle connection. 
	The union of the domains of all the poles that are connected to the original one by a chain of saddle connections has therefore an empty boundary, hence it is all of $\Cal S$. 
\end{proof}

Proposition~\ref{prop:fundamentalgroupoid} makes good sense only if $\Sigma\neq\emptyset$, which can always be assumed potentially after adding to it a pole of order 0
in the case of degree $2$ vector fields.

The \emph{relative homology}  $H_1(\Cal S,\Sigma;\Z)$, which is the abelianization of the fundamental groupoid  $\pi_1(\Cal S,\Sigma)$, is therefore also generated by the saddle connections.

\begin{definition}
The \emph{period map} $\nu: H_1(\Cal S,\Sigma;\Z)\to\C$ is given by integration of the time form $dt=\tfrac{Q(z)}{P(z)}dz$
along cycles.  The \emph{period} of a saddle connection $\gamma$ is
$$\nu_\gamma= \int_\gamma\frac{Q(z)}{P(z)}\;dz\neq 0.$$
\end{definition}

\begin{remark}
Even for generic vector fields with all zeros simple, the knowledge of the multiplicities of its poles in $\Sigma$ and of the period map $\nu: H_1(\Cal S,\Sigma;\Z)\to\C$ alone does not determine the translation model. Another combinatorial information expressing the organization of the saddle connections is necessary. In case of polynomial vector fields, the cyclic order in which the periodic domains are attached to the poles provides such information.
\end{remark}

\subsection{The periodgon and star domain}

In order to understand the form of the translation model we shall cut it and unwrap it into a flat simply connected 
domain on the surface of $t(z)$: the star domain, and its core part: the periodgon, defined below.
\\

\begin{definition}~ \label{def:periodgon} 
Consider a rational vector field of degree at least 3, and let $\Sigma$ be the set of poles of positive multiplicity only.
	\begin{enumerate} [wide=0pt, leftmargin=\parindent]
\item \emph{Cuts}:
\begin{itemize}
	\item For each equilibrium $z_j$, choose an end of the periodic/parabolic domain (several choices may be possible) and cut along a separatrix of \eqref{rotated_beta} inside the domain connecting $z_j$ to the end. If $\nu_j\neq 0$, e.g. if $z_j$ is simple, let $\beta=\arg\nu_j+\frac{\pi}{2}$.
	\item For each annular domain cut along (one of) the shortest saddle connection(s) inside the domain (not on the boundary).
	\item For each component of the complement of all the periodic/parabolic/annular domains with at least two poles, keep adding cuts each time along (one of) the shortest heteroclinic saddle connection(s) not intersecting the previous cuts, until 
	the cuts form a tree connecting all the poles. 
\end{itemize}

\item The complement in $\CP^1$ of these cuts is a simply connected domain, the closure of whose image by (a branch of) $t(z)$ \eqref{eq:t} is called the \emph{star domain}.

\item The \emph{periodgon} is  the part of the star domain that corresponds to the complement of all the periodic and parabolic domains. 

\item If in the definition of the star domain and the periodgon we choose instead any set of non-intersecting cuts simply connecting all the poles then the corresponding objects will be called a \emph{generalized star domain} and a \emph{generalized periodgon}.
\end{enumerate} \end{definition}
\bigskip

\begin{remark}~ 
\begin{enumerate}[wide=0pt, leftmargin=\parindent] 
\item If a periodic/parabolic domain has several ends, or if several shortest saddle connections appear in the construction of the cuts, then several choices are possible, leading to several different  periodgons describing the same dynamics. But this is a nongeneric case.	
\item Heteroclinic connections between the poles of a rotated vector field \eqref{rotated_beta} always exist for some $\beta$ because of the monotonous movement of the separatrices of a family of rotated vector fields when $\beta$ varies. See Proposition~\ref{prop:fundamentalgroupoid}.
\item  The (generalized) periodgon is a compact polygonal domain, possibly degenerate, on the translation surface of $t(z)$ with vertices at the conical points.
The projection of the  (generalized) periodgon on $\C$-space may have self-intersections. This is because the translation surface of $t(z)$ is ramified at the images of the poles. 
\item In the case when equilibria are simple, the boundary of the (generalized) periodgon is formed by $d$ segments corresponding to the period vectors $\nu_1,\dots, \nu_d$ of the $d$ 
equilibrium points $z_1,\ldots,z_d$ and  $(n-1)$ pairs of parallel equal vectors of opposite directions corresponding to the travel times for the cuts between the $n$ poles. And the (generalized) star domain is the union of the periodgon and $d$ infinite branches of respective width $\nu_1, \dots, \nu_d$, which are orthogonal to the respective sides  $\nu_1, \dots, \nu_d$.
\item Note that the (generalized) periodgon has no limit in a parametric family when approaching a parabolic point. Indeed, the sides of the periodgon corresponding to the merging points become infinite and their arguments turn with the parameter.
For example, in the family of the form  $\dot z = z^2 -\eps z+O(z^3)$ (resp. $\dot z = z^2-\eps +O(z^3)$) the merging sides are $\sim \pm\frac{2\pi i}{\epsilon}$
(resp. $\sim \pm\frac{2\pi i}{\sqrt\epsilon}$), while their sum has a finite limit given by the period of the limit parabolic point. 
\end{enumerate}\end{remark}

\noindent{\bf Coming back to Example~\ref{example1}.} We consider the system 
\eqref{vf_example}. In Figure~\ref{multiple_cuts}. There are infinitely many  saddle connections between the two poles, of which only the two arcs in $\R\cup\infty$ are of the shortest length and can each serve as a cut  (see Figure~\ref{multiple_cuts}), which yields to two different periodgons. Indeed, the left and right  sides of the periodgon are glued together, forming a cylinder, on which we  have saddle connections that make an arbitrary number of turns around the cylinder.

\begin{figure} \begin{center}
\subfigure[Phase portrait]{\includegraphics[height= 4.5cm]{Phase_cuts}}\qquad\subfigure[Periodgon and star domain]{\includegraphics[height= 4.5cm]{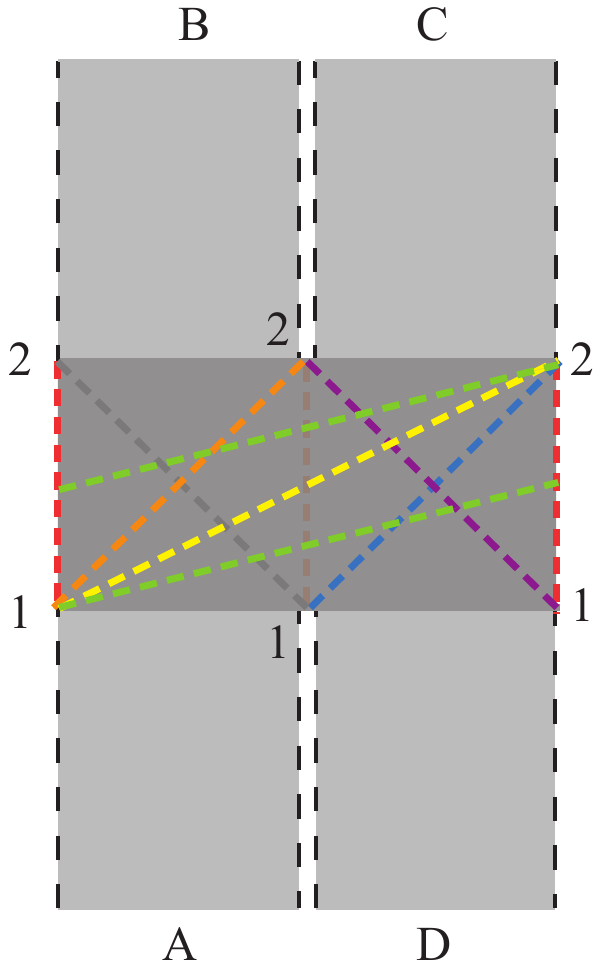}}\caption{In (a), six different saddle connections  between the two poles in the system \eqref{vf_example}; only two of them have the shortest length and can serve as a cut. 
In (b), the star domain is in gray and contains the periodgon in darker gray. On it appear the corresponding  saddle connections on the periodgon, plus two more. In particular the green one spirals around the tubular part of the time surface.}\label{multiple_cuts}\end{center}
 \end{figure}

\begin{example} We consider the system 
\begin{equation}\dot z = i\frac{z^4+1.2}{1-z^2},\label{vf_example2}\end{equation} which is a deformation of \eqref{vf_example}, and whose periodgon is a deformation of the periodgon of Figure~\ref{multiple_cuts}(b) corresponding to the red cut in Figure~\ref{multiple_cuts}(a).  The four centers have now become attracting or repelling foci and we still have the two poles at $z=\pm1$. Moreover, the imaginary axis is still invariant and still belongs to a family of periodic orbits. 
The phase portrait appears in Figure~\ref{fig:example2}(a), and the periodgon and star domain in Figure~\ref{fig:example2}(b). 
\begin{figure} \begin{center}
\subfigure[Phase portrait]{\includegraphics[height= 4.5cm]{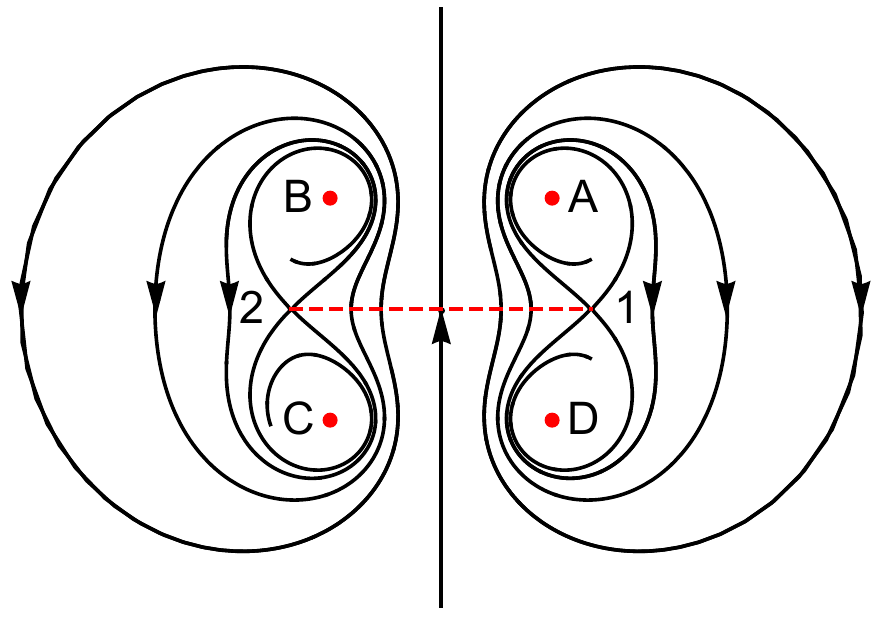}}\qquad\subfigure[Periodgon and star domain]{\includegraphics[height= 4.5cm]{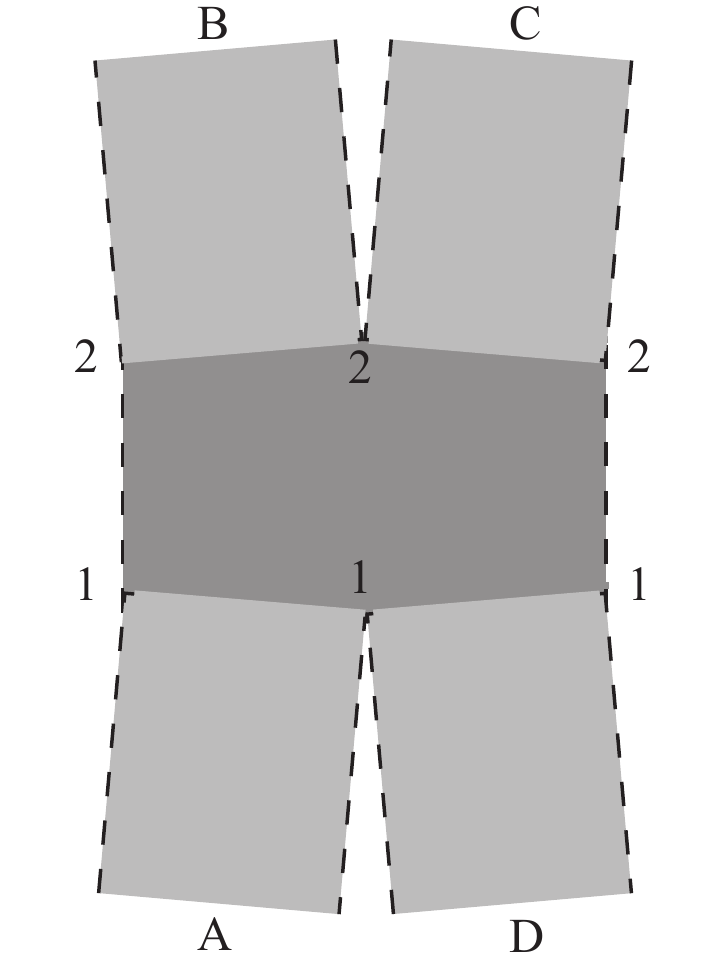}}
\caption{The system \eqref{vf_example2}. There is still an annulus of periodic solutions on the sphere $\CP^1$, one of which is the imaginary axis. The star domain is in gray and contains the periodgon in darker gray.}\label{fig:example2}\end{center} \end{figure}\end{example}

\subsection{Bifurcations of the periodgon}
An important question is to understand how the periodgon and star domain depend on the coefficients of $P$ and $Q$ in a rational vector field of degree $d$ and what are their bifurcations. 
We address this question when $d=4$ in the next section.

\section{The case of a rational vector field of degree $4$}\label{sec:deg4}
Using an affine change of coordinate we can suppose that the vector field has the form 
\begin{equation}\dot z = \frac{z^4+\eta_3z^3+\eta_2z^2+\eta_1z+\eta_0}{1-az^2}.\label{vf_4}\end{equation}
There is still one degree of freedom coming from Moebius transformations that can be used to simplify the system: see Proposition~\ref{prop:normal_form} below. The system is truly rational when $a\neq0$. The periodgon was described in \cite{KR} in the particular case $a=\eta_2=\eta_3=0$,
and in \cite{R19} in the  case $a=\eta_0=\eta_3=0$. We will start by investigating the case $a\neq0$. As a second step we will consider what happens when we let $a=0$.

\subsection{Structure theorem}

\begin{proposition}\label{prop:configuration_4} We consider the vector field \eqref{vf_4} with $a\neq0$.
\begin{enumerate}
\item The only generic configurations of periodic domains are $(3,1)$ and $(2,2)$, namely $3$ (resp. $2$) periodic domains attached to one pole and $1$ (resp. $2$) attached to the other pole. Moreover, these configurations do occur.  
\item A periodic orbit either belongs to the basin of a center, or surrounds on each side two singular points and one pole and belongs to an annulus of periodic orbits. The latter case can only occur in the case of a configuration $(2,2)$. \end{enumerate}
\end{proposition}
\begin{proof} 
\begin{enumerate}[wide=0pt, leftmargin=\parindent]
\item An end of a periodic domain at a simple pole has an opening of $\frac{\pi}2$ in the $z$-coordinate. 
Let us show that it is impossible to have four periodic domains attached to one pole. Indeed, for this to occur, the  four ends of the  four periodic domains would need to be all tangent at the pole, and hence  the four domains would need to appear simultaneously for the same angle of rotation.  But the homoclinic loops bounding four periodic domains cannot occur simultaneously, since a simple pole has only four separatrices. 

\item This follows from Lemma~\ref{lemma:configurations}. 
\end{enumerate}
\end{proof}

We will show that the different types of generic configurations of Proposition~\ref{prop:configuration_4} indeed exist in the family \eqref{vf_4}.

\begin{theorem}\label{thm:types} 
In the case of distinct zeros and poles we have generically three types of periodgon, or equivalently, three geometric types of  translation model (see Figure~\ref{pgon_types}):

\begin{itemize} 
\item A periodgon corresponding to configuration $(3,1)$;
\item A periodgon corresponding to configuration $(2,2)$ with no annular domain, noted $(2,2)_{NO}$;
\item A periodgon corresponding to configuration $(2,2)$ with an annular domain, noted $(2,2)_{AD}$. 
\end{itemize}
In all cases the periodgon is planar, i.e. its identical projection from the translation surface of $t$ to $\C$ has no self-intersections. It contains two parallel sides corresponding to a cut between the two poles.   

These three types of periodgon can all bifurcate from the polynomial case $\dot z = P(z)$ when the two poles merge at infinity. In the merging the parallel sides corresponding to the cuts shrink to a point.\end{theorem}

\begin{proof}
The existence of the three configurations within the general family of degree 4 rational vector fields with simple poles will follow from the Realization Theorem (Theorem~\ref{thm:realization}) and Figure~\ref{pgon_types}. 
The fact that the (generalized) periodgon has no self-intersection is because of the small number of sides which, when negatively oriented, are of the form 
$\nu_{\sigma(1)}, \nu_{\sigma(2)}, \tau, \nu_{\sigma(3)}, \nu_{\sigma(4)}, -\tau$ in a cyclic order for the configuration $(2,2)$, resp. $\nu_{\sigma(1)}, \nu_{\sigma(2)},  \nu_{\sigma(3)}, \tau,\nu_{\sigma(4)}, -\tau$ for the configuration $(3,1)$, where $\sigma$ is some permutation on the indices, and the fact that the periodgon is bounded.
\end{proof}

\begin{figure} \begin{center}
\subfigure[$(3,1)$]{\includegraphics[height=3.2cm]{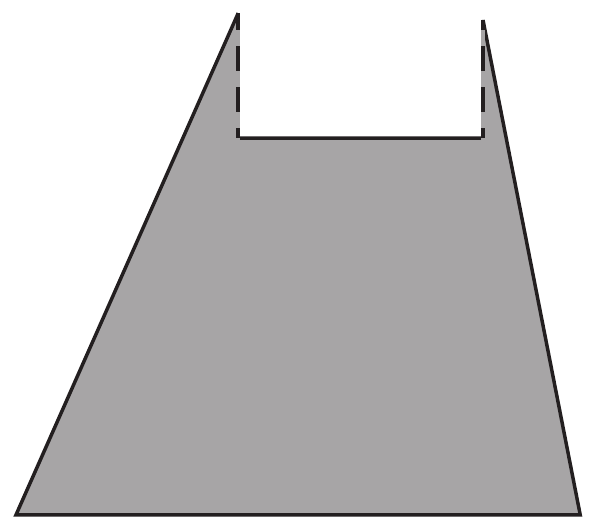}}\qquad\qquad\subfigure[$(2,2)_{NO}$]{\includegraphics[height=3.2cm]{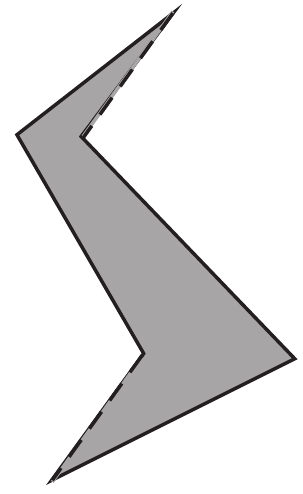}}\qquad\qquad \subfigure[$(2,2)_{AD}$]{\includegraphics[height=3.2cm]{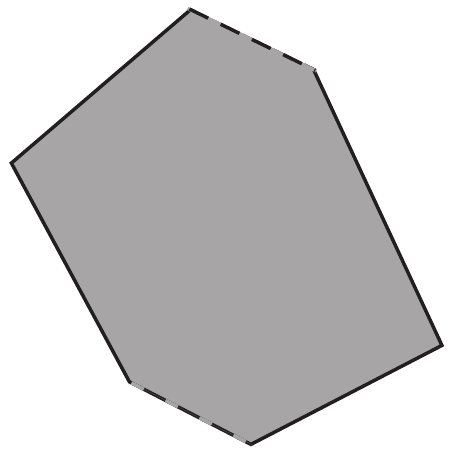}}
\caption{The three types of periodgon of Theorem~\ref{thm:types} associated to the respective configurations.}\label{pgon_types}\end{center}\end{figure}

\subsection{Projective reduction of the family \eqref{vf_4}} 

Simpler forms than \eqref{vf_4} can be found when using projective transformations to simplify the system. 

\begin{proposition}\label{prop:normal_form} 
\begin{enumerate}[wide=0pt, leftmargin=\parindent] 
\item Any system \eqref{vf_4} with $a\neq0$ can be brought by a Moebius transformation depending analytically on the coefficients $(a, \eta_0, \dots, \eta_3)$ to the form \begin{equation} \dot z =C(\eps) \frac{z^4+\eps_3z^3+\eps_2z^2+\eps_1z+\eps_0}{1-z^2},\label{vf_4bis}\end{equation} 
for some $C(\eps)\neq0$.
\item Let  \eqref{vf_4} be a family with $a\neq0$ depending on the multi-parameter  $\eta=( \eta_0, \dots, \eta_3)$ varying in the neighborhood of the origin in $\C^4$.Then there exists a family of Moebius transformations depending analytically on $\eta$ and transforming the family  to the form 
\begin{equation} \dot z =C(\eps) \frac{z^4+\eps_2z^2+\eps_1z+\eps_0}{1-z^2},\label{vf_4quat}\end{equation}  
for some $C(\eps)=1+O(\eps)$. 
\item Let  \eqref{vf_4} be a family depending on the multi-parameter $\eta=(a, \eta_0, \dots, \eta_3)$ varying in the neighborhood of the origin in $\C^5$, and let $C\in \C^*$. Then there exists a family of Moebius transformations depending analytically on $\eta$ and transforming the family  to the form 
\begin{equation} \dot z =C\frac{z^4+\eps_2z^2+\eps_1z+\eps_0}{1-\delta z^2},\label{vf_4ter}\end{equation} 
where $\delta$ and the $\eps_j$ depend analytically on $\eta$ and vanish for $\eta=0$. 
\end{enumerate}
\noindent In (1) and (2), using a real (complex) scaling of time depending real-analytically (resp. analytically) on $\eta$ we take make $C= \exp(i\beta(\eps))$ (resp. $C=1$ or any desired value of $C$).

\end{proposition} 

\begin{proof} We consider a change of coordinate $Z=d\frac{z+b}{1+cz}$ with $d(1-bc)\neq0$. The transformed system has poles of the form $\pm Z_0$ if and only if $c=ab$. 
\begin{enumerate} \item Moreover $Z_0=1$, if and only if $a-d^2=0$.  It suffices to take $d=\pm \sqrt{a}\neq0$. 

\item 
By (1) we can start with a system of the form \eqref{vf_4bis}, and consider a change $Z=\frac{z+b}{1+bz}$.
The coefficient of $Z^3$ has the form $h(\eta)(b+O(\eta))$, where $h(0)\neq0$. Hence, by the implicit function theorem it will vanish for some analytic function $b(\eta)$. 
\item We consider a change $Z=d\frac{z+b}{1+abz}$. The coefficient of $Z^4$ is of the form $\frac{1+O(\eta)}{d^3}$, and hence can be taken equal to $C$ for some nonzero analytic function $d(\eta)$. Then the coefficient of $Z^3$ vanishes for $-4bd+ O(\eta)=0$, which can be solved by the implicit function theorem. 
\end{enumerate}
 \end{proof}

\subsection{An example}

We consider the following subfamily:
\begin{equation}\dot z =i \frac{(z^2-\eta_1)(z^2+\eta_2)}{1-z^2},\label{centre_organisateur}\end{equation} 
with $\eta_1,\eta_2\in \R^+$ and $\eta_1>1$ (see Figure~\ref{centre_org}(a)). The system is symmetric with respect to $i\R$ and reversible with respect to $\R$. The singular points $A =\sqrt{\eta_1}$ and $C =-\sqrt{\eta_1}$ are centers and their periodic domains are each attached to one pole, namely $S_1 = 1$ or $S_2 = -1$. The periodic domains of $B=  i \sqrt{-\eta_2}$ and $D=-i\sqrt{-\eta_2}$ are symmetric with respect to $i\R$. Hence they both are bounded by a heteroclinic cycle through the two poles, one side of the heteroclinic cycle being $[-1,1]\subset \R$ (see Figure~\ref{centre_org}(b)).
\begin{figure}\begin{center}
\subfigure[The vector field \eqref{centre_organisateur}]{\includegraphics[width=5cm]{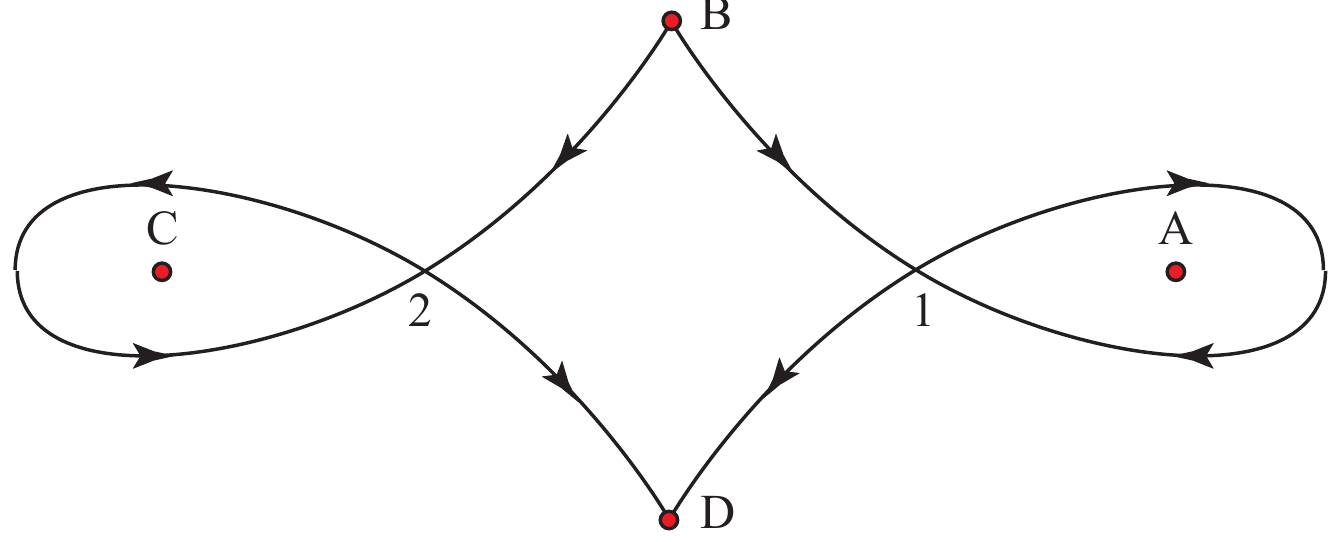}}\qquad \subfigure[The periodic domains]{\includegraphics[width=5cm]{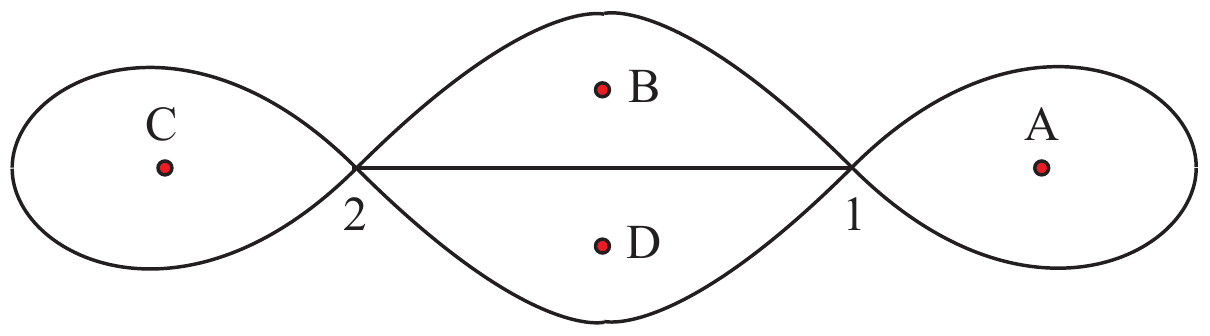}}\caption{The vector field \eqref{centre_organisateur} and the periodic domains of the singular points.}\label{centre_org}\end{center}\end{figure}
There are four ways to draw the periodgon depending on the four ways we take the cuts from each of $B,D$ to one pole $S_j$ (see Figure~\ref{four_ways}), to which correspond four star domains and periodgons (see Figure~\ref{four_periodgons}).

\begin{figure}\begin{center}
\subfigure[]{\includegraphics[width=5cm]{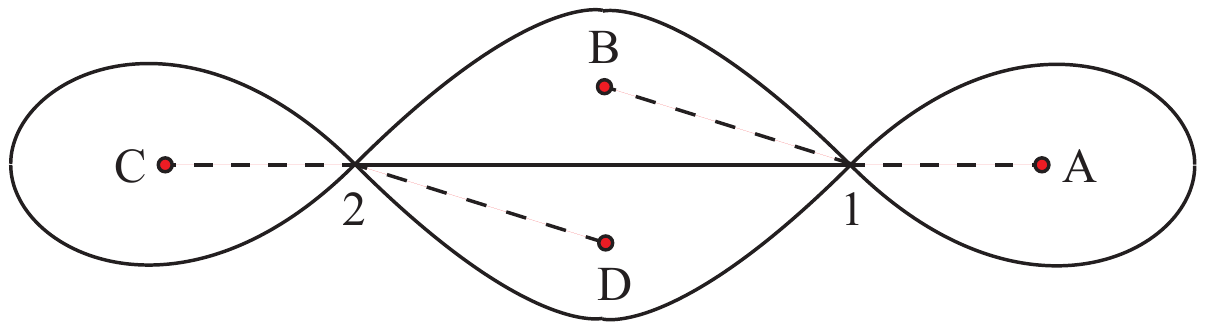}}\quad \subfigure[]{\includegraphics[width=5cm]{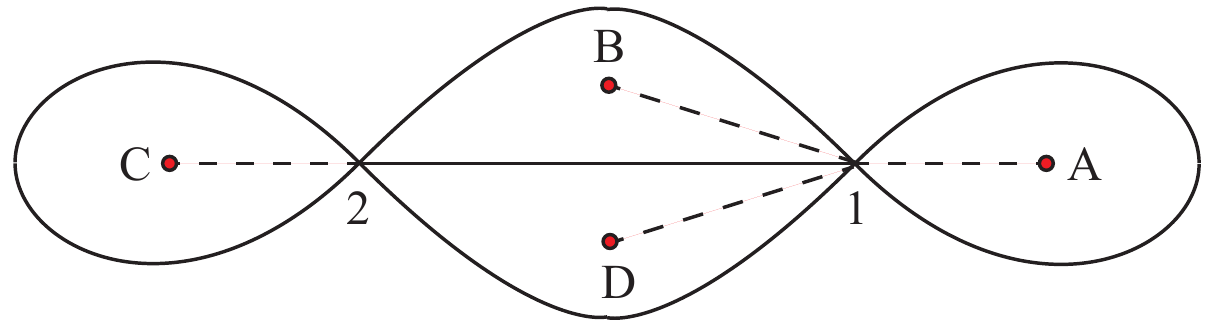}}\\ \subfigure[]{\includegraphics[width=5cm]{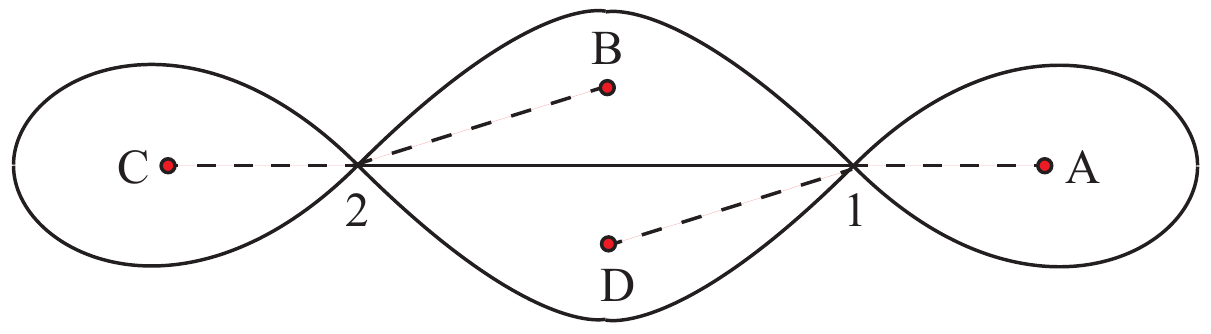}}\quad\subfigure[]{\includegraphics[width=5cm]{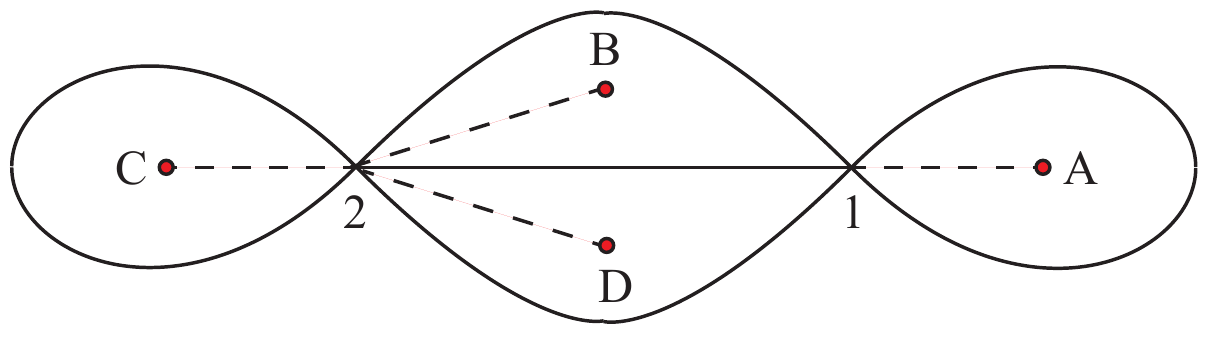}}\caption{The four ways to take the cuts for the system \eqref{centre_organisateur} and its periodic domains.}\label{four_ways}\end{center}\end{figure}

\begin{figure}\begin{center}
\subfigure[]{\includegraphics[width=5cm]{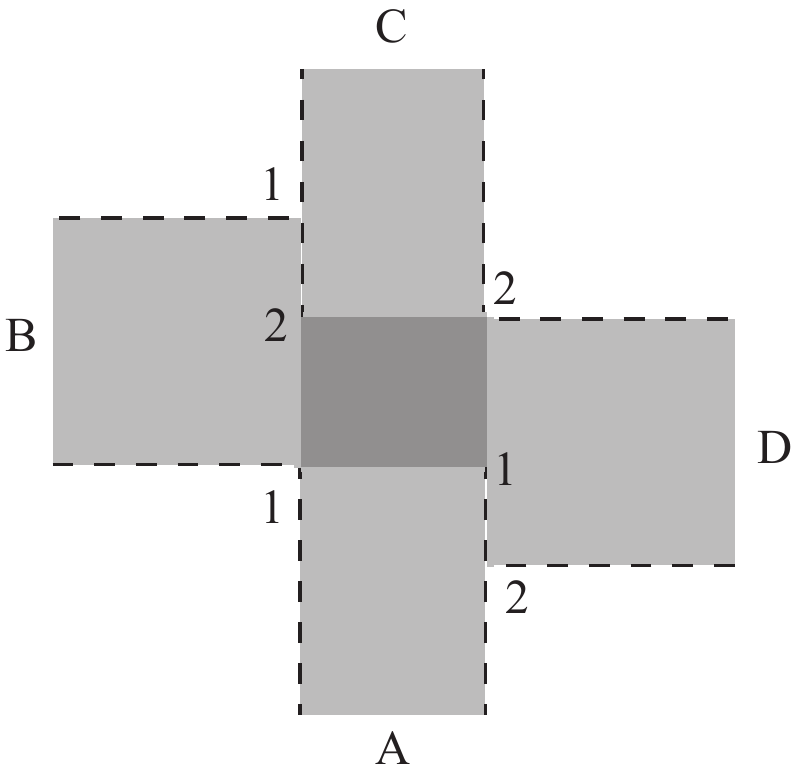}}\quad \subfigure[]{\includegraphics[width=5cm]{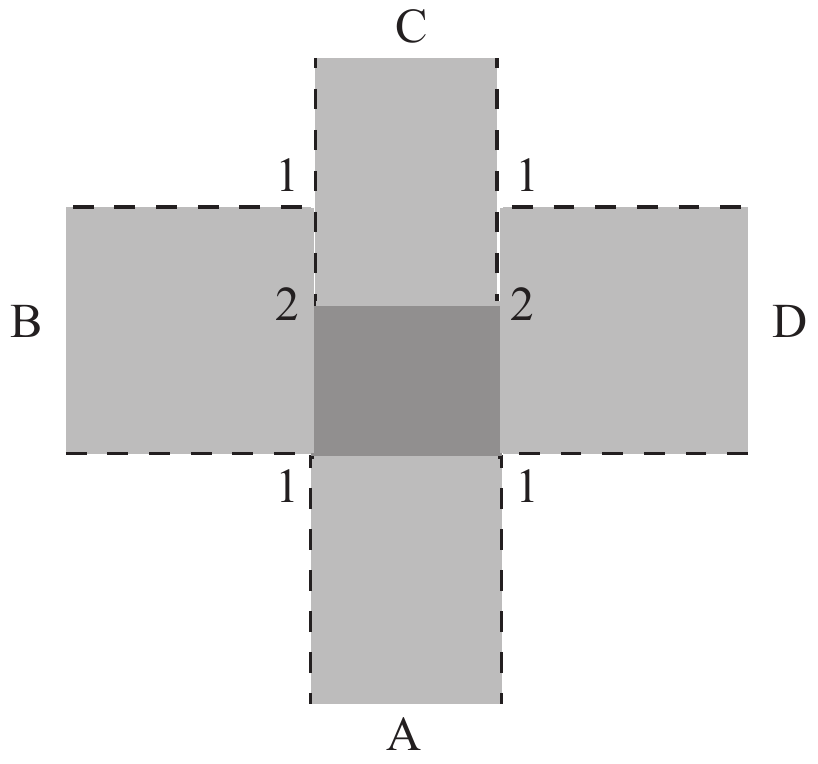}}\\ \subfigure[]{\includegraphics[width=5cm]{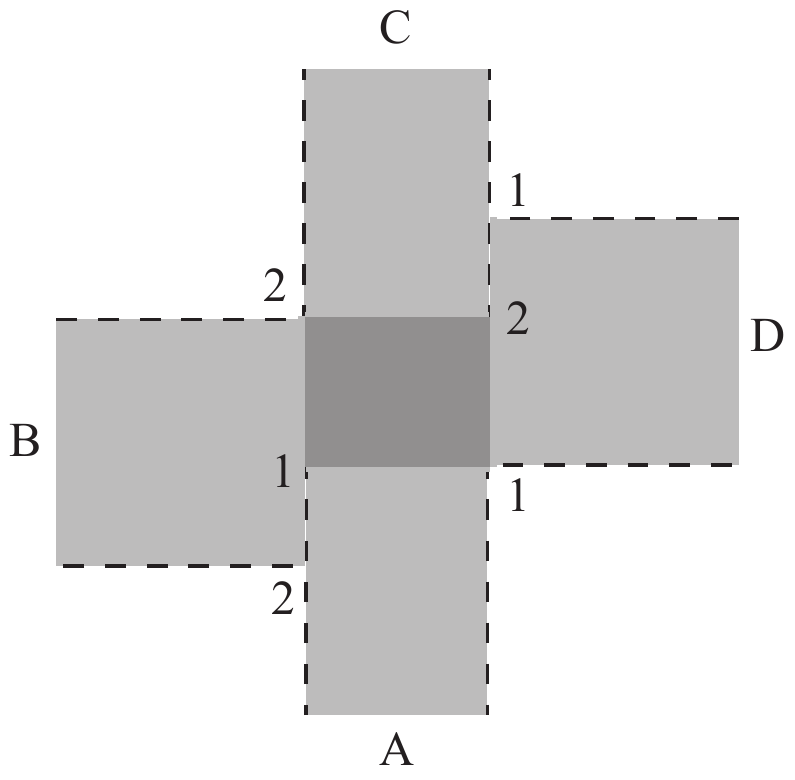}}\quad\subfigure[]{\includegraphics[width=5cm]{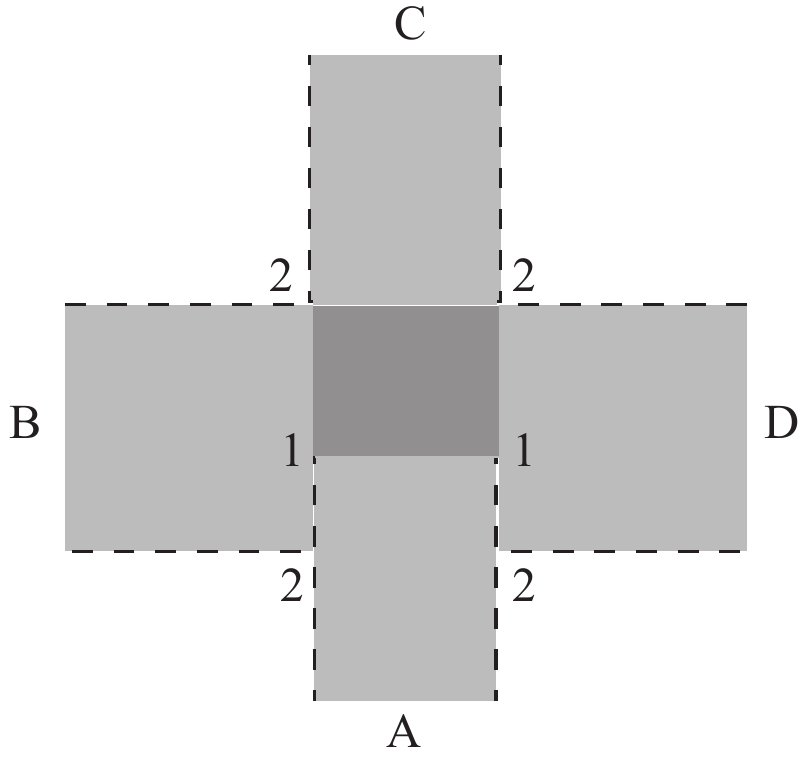}}\caption{The four star domains corresponding to the four ways to take the cuts for the system \eqref{centre_organisateur}. The interior of the periodgon in each case is in darker gray.}\label{four_periodgons}\end{center}\end{figure}

A perturbation breaking the symmetry with respect to $i\R$, while keeping the reversibility with respect to $\R$, for instance
$$\dot z =i \frac{(z^2+\eta_3z -\eta_1)(z^2+\eta_2)}{1-z^2},$$
with $\eta_3$ real small, will then create the configurations $(3,1)_1$ (with the periodic domains of $B,D$ attached to $S_1$) or $(3,1)_2$ (with  the periodic domains of $B,D$ attached to $S_2$) depending on the sign of $\eta_3$. 
A second perturbation by moving $\eta_2$ in \eqref{centre_organisateur} outside $\R$ will keep the symmetry with respect to the origin but will break both the symmetry with respect to $i\R$ and the reversibility with respect to $\R$. It will then create the  two configurations $(2,2)$, where either  the periodic domain of $B$ is attached to $S_1$ and that of $D$ to $S_2$, or the converse. All these configurations appear in Figure~\ref{perturbed_pd}. The corresponding four types of periodgons (possibly with additional deformations) are illustrated in Figure~\ref{deformed_periodgon}.
\begin{figure}\begin{center}
\subfigure[]{\includegraphics[width=5cm]{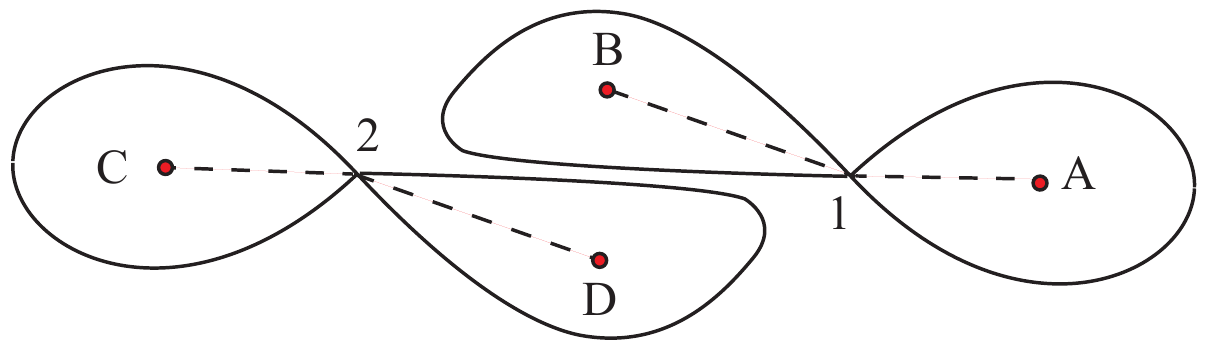}}\qquad \subfigure[]{\includegraphics[width=5cm]{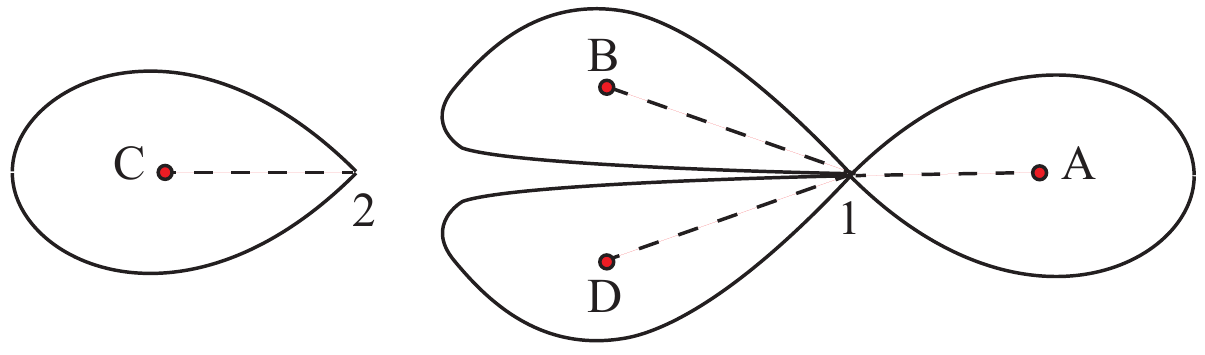}}\\ \subfigure[]{\includegraphics[width=5cm]{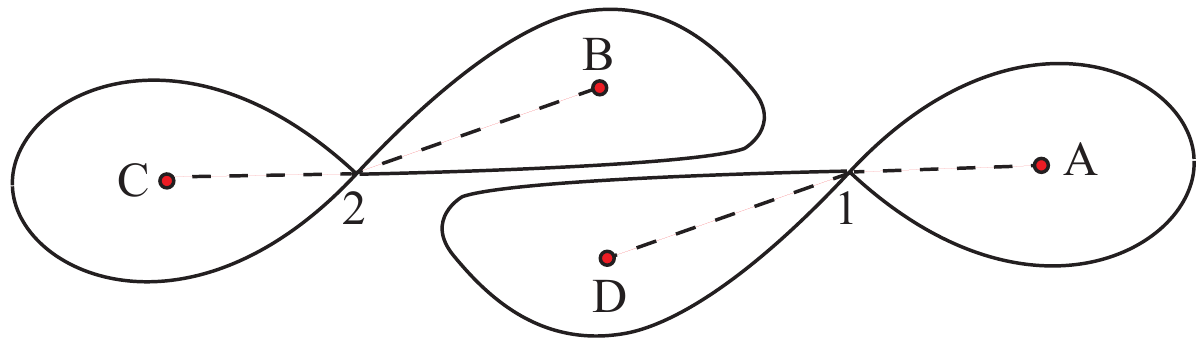}}\qquad\subfigure[]{\includegraphics[width=5cm]{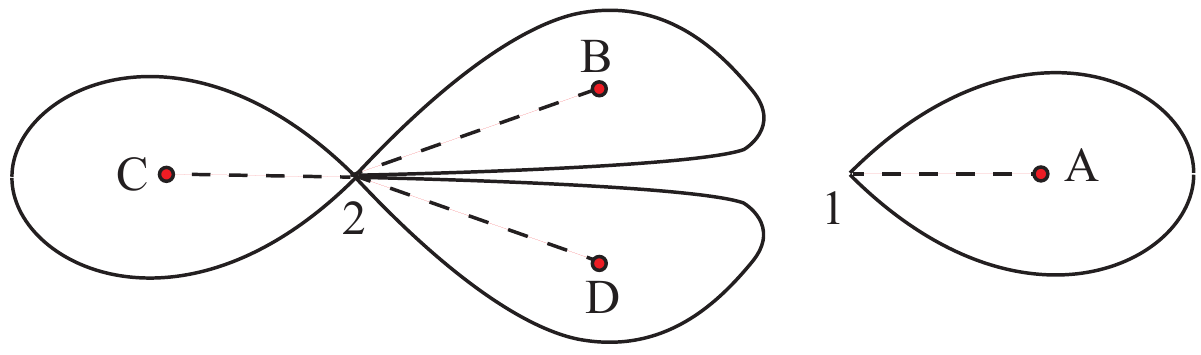}}\caption{The deformed periodic domains occurring in perturbations of \eqref{centre_organisateur}.}\label{perturbed_pd}\end{center}\end{figure}
\begin{figure}\begin{center}
\subfigure[]{\includegraphics[width=5cm]{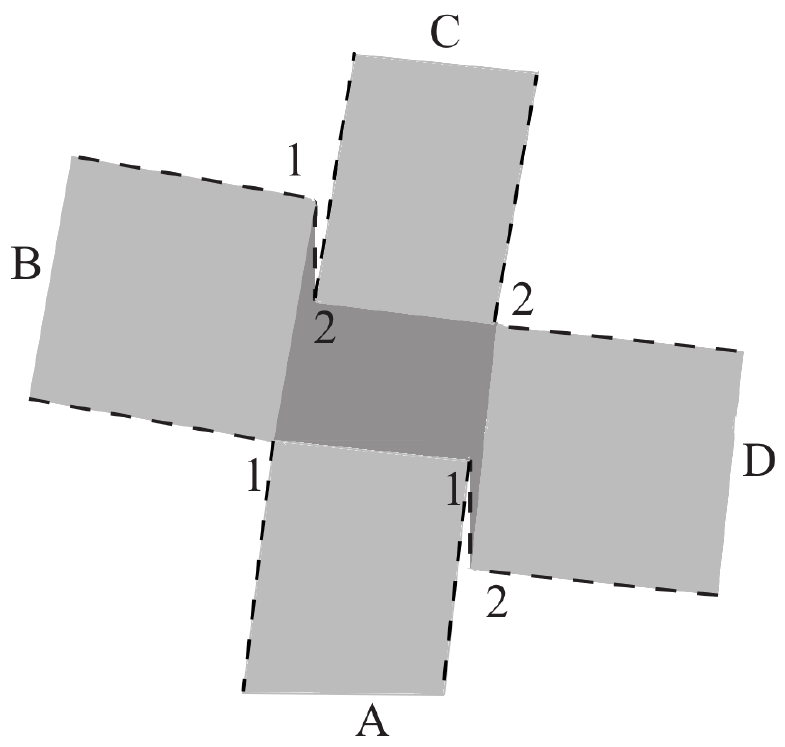}}\qquad \subfigure[]{\includegraphics[width=5cm]{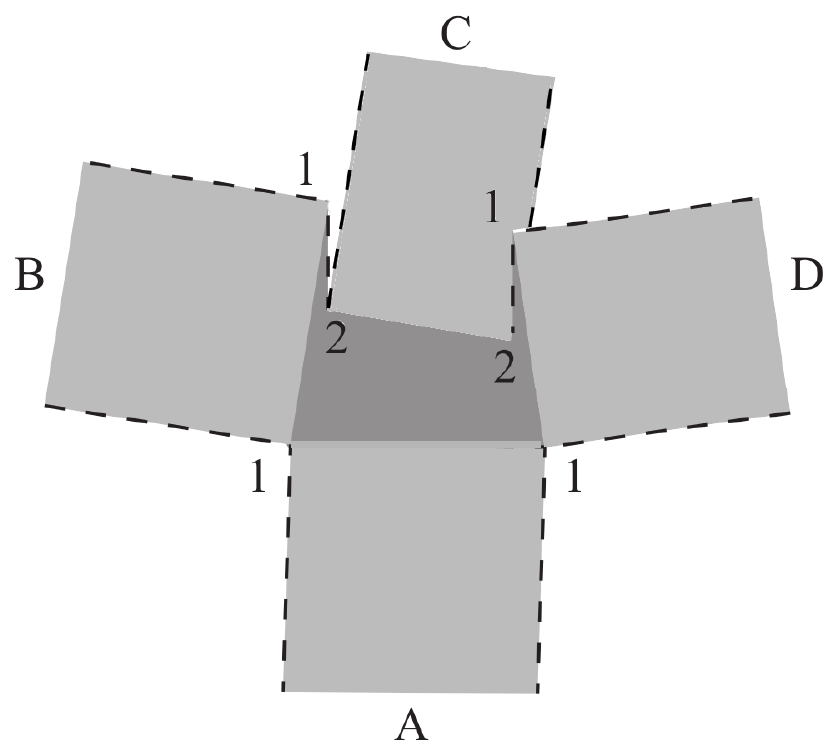}}\\ \subfigure[]{\includegraphics[width=5cm]{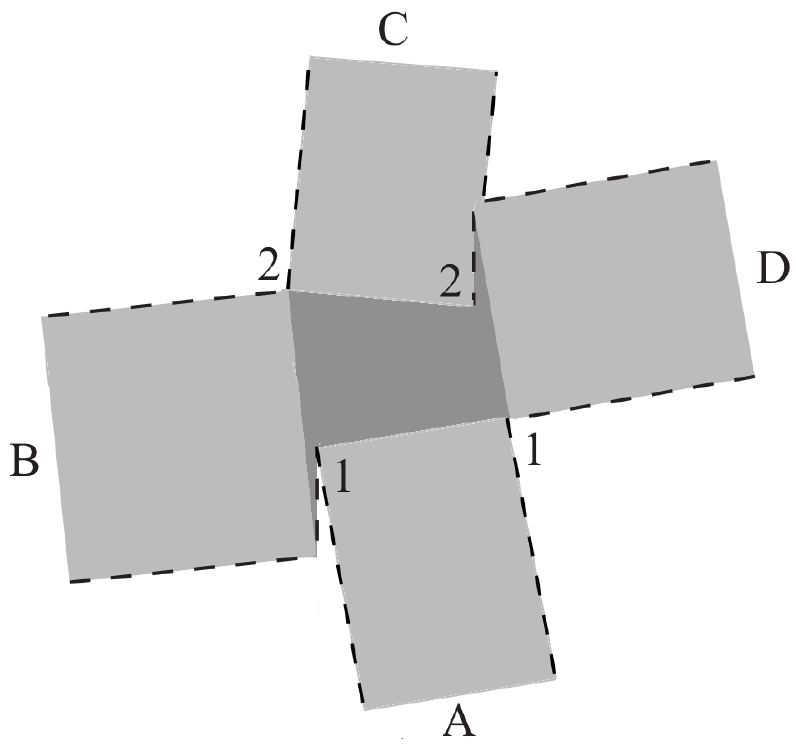}}\qquad\subfigure[]{\includegraphics[width=5cm]{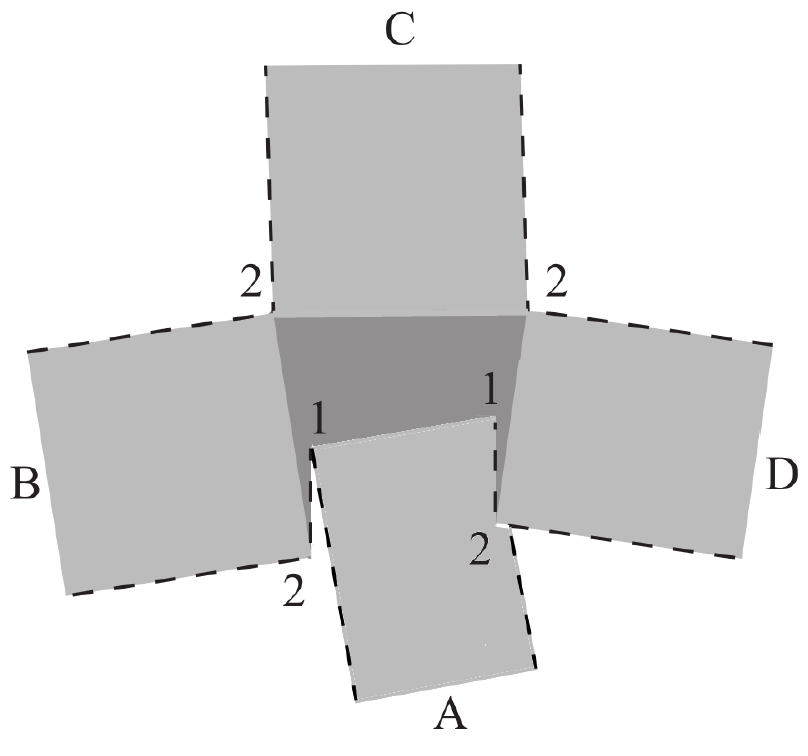}}\caption{The four types of periodgons and star domains occurring  in perturbations of \eqref{centre_organisateur}.
}\label{deformed_periodgon}\end{center}\end{figure}
We see the general form of the periodgon: it has four segments corresponding to homoclinic loops around the four singular points and two segments corresponding to a cut between the poles.

\subsection{Bifurcation between the two configurations $(3,1)$ and $(2,2)$} 
In the family \eqref{centre_organisateur} and Figure~\ref{centre_org} we have seen a bifurcation between the two configurations $(3,1)$ and $(2,2)$. But this was a codimension $2$ phenomenon where two singular points had their periodic domains bounded by heteroclinic loops through the two poles. The phenomenon however is of ocodimension $1$ and a bifurcation occurs as soon as one singular point has a periodic domain bounded by a heteroclinic loop through the two poles (see Figure~\ref{bif_conf}). The bifurcation is easily described through the change of the periodgon (see Figure~\ref{bif3-1_2-2}).
\begin{figure}\begin{center}
\subfigure[]{\includegraphics[width=5cm]{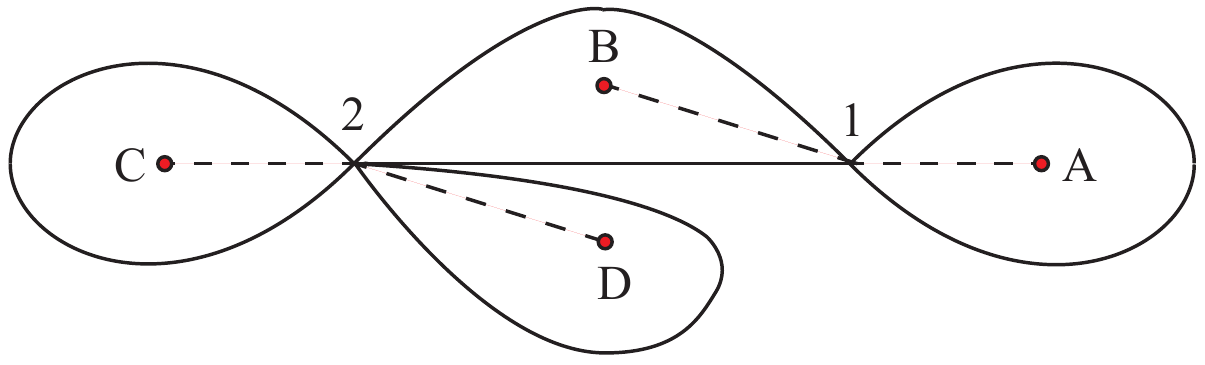}}\qquad \subfigure[]{\includegraphics[width=5cm]{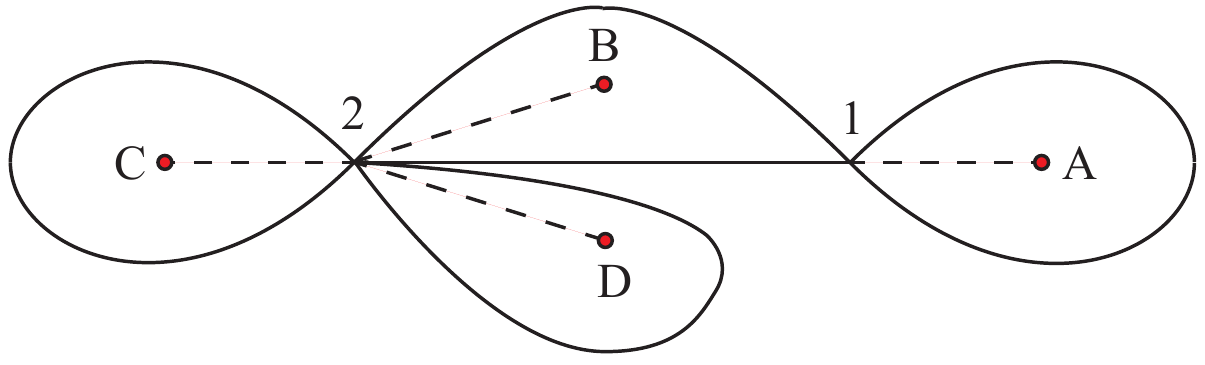}}\\ \subfigure[Deformation $(2,2)$]{\includegraphics[width=5cm]{deform_1}}\qquad\subfigure[Deformation $(3,1)$]{\includegraphics[width=5cm]{deform_4}}\caption{The periodic domains in a generic bifurcation between the configurations $(3,1)$ and $(2,2)$ and the two ways to choose the cuts leading to the deformations  in (c) and (d).}\label{bif_conf}\end{center}\end{figure}

\begin{figure}\begin{center}
\subfigure[]{\includegraphics[height=5cm]{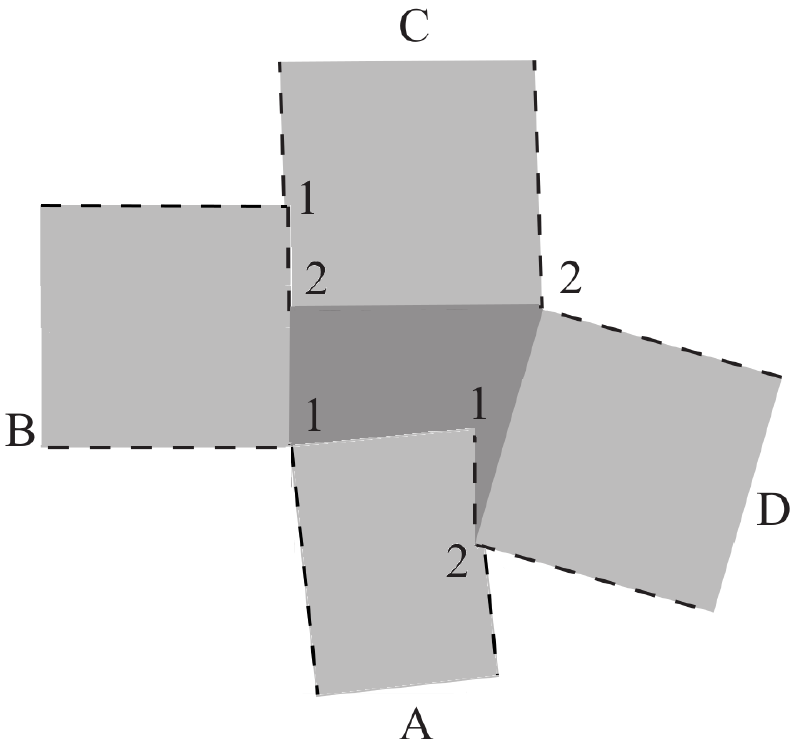}}\qquad \subfigure[]{\includegraphics[height=5cm]{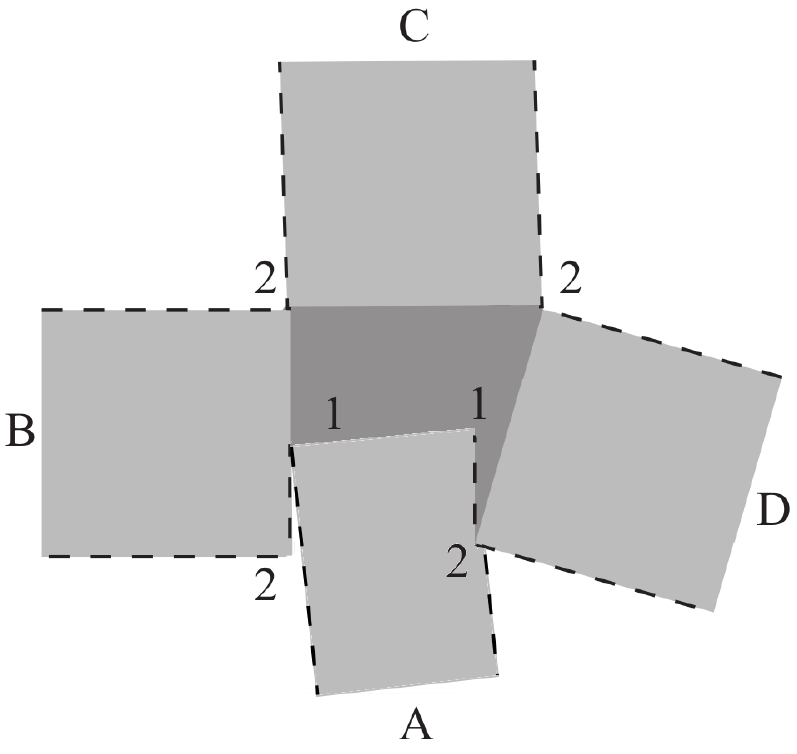}}\\ 
\subfigure[Deformation $(2,2)$]{\includegraphics[height=5cm]{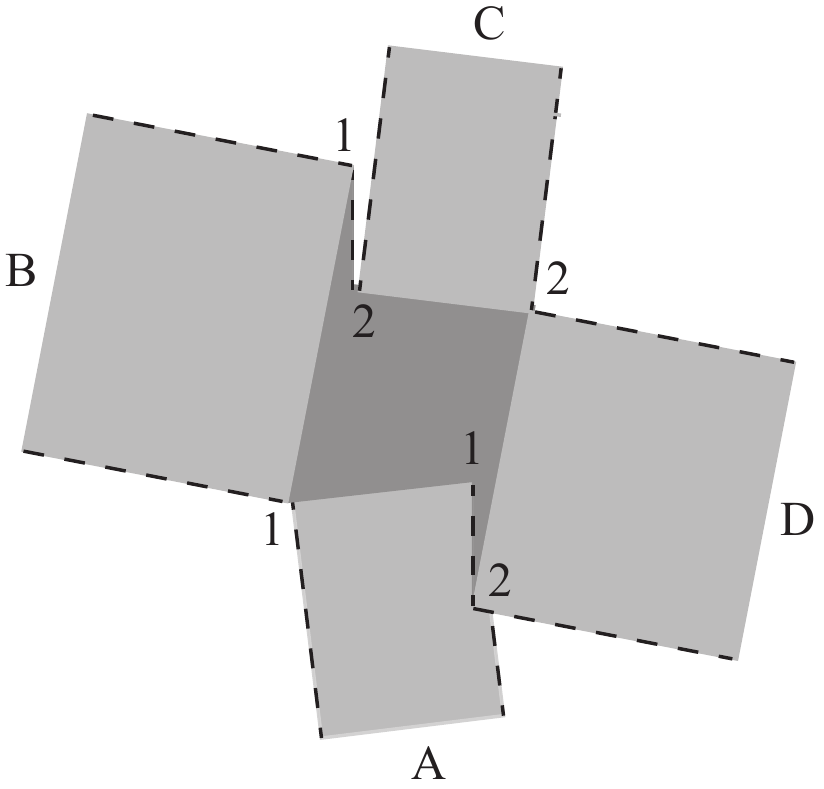}}\qquad\subfigure[Deformation $(3,1)$]{\includegraphics[height=5cm]{deformed_periodgon4x}}\caption{The periodgon and star domains in the generic bifurcation between the configurations $(3,1)$ and $(2,2)$ of Figure~\ref{bif_conf}. In (a) and (b) the same bifurcation situation with two choices of periodgon depending on the cuts and in (c) and (d) the deformations.
}\label{bif3-1_2-2}\end{center}\end{figure}

\subsection{Bifurcation creating an annular domain}

When considering a rotated vector field $\dot z = e^{i\alpha} \frac{P(z)}{Q(z)}$ of $\dot z =  \frac{P(z)}{Q(z)}$, its corresponding periodgon is obtained by rotating the periodgon of $\dot z =  \frac{P(z)}{Q(z)}$ of an angle $-\alpha$. A rotated vector field $\dot z = e^{i\alpha} \frac{P(z)}{Q(z)}$ has an annulus of periodic solutions  if there exists inside the periodgon a continuum of parallel segments joining the two parts of the boundary corresponding to the cut between the poles. Such a continuum does not exist in any of the periodgons of  Figure~\ref{deformed_periodgon}. When moving further from \eqref{centre_organisateur}, in the case of the configuration $(2,2)$, the periodgon will be deformed. Is it possible to have a continuous deformation leading smoothly to the formation of an annular domain (for instance as in Figures~\ref{multiple_cuts} and ~\ref{fig:example2})? This is the phenomenon that we now study. 

We consider a vector field \eqref{vf_4} with simple singular points and poles, and such that each periodic domain is  attached to only one pole. 
Let $\nu_1, \nu_2,\nu_3,\nu_4$ be the four period vectors of the singular points of \eqref{vf_4} and let $\pm\tau$ be the period of the chosen saddle connection between the two poles along which one cuts. Note that $\nu_1+\nu_2+\nu_3+\nu_4=0$.
There exists a permutation $\sigma$ of $\{1, 2, 3, 4\}$ such that the periodgon has, for configuration $(2,2)$, sides given by the vectors
$$\nu_{\sigma(1)}, \nu_{\sigma(2)}, \tau, \nu_{\sigma(3)}, \nu_{\sigma(4)}, -\tau,$$
in this circular order.  Note that the $\tau$ vector can never be aligned with two  $\nu_{\sigma(j)}$ located between the two $\tau$. Indeed, a simple pole has four separatrices. Hence if for some $\beta$ the vector field \eqref{rotated_beta} has a heteroclinic connection, it can have simultaneously at most one homoclinic connection.  

We want to characterize those realizable periodgons for which the realized vector field has an annular domain. This comes to characterize geometrically those periodgons for which there exists, inside the periodgon, a continuum of parallel segments joining the two segments $\tau, -\tau$ of the boundary. Each part of the boundary of  an annulus of periodic solutions not surrounding a center can be of two kinds: \begin{itemize} 
\item a single homoclinic connection; 
\item a pair of homoclinic connections forming a figure eight (as in Figure~\ref{fig:example1}).
\end{itemize} 

\begin{proposition}\label{prop:ann_domain} The vectors $
v_{12}=\nu_{\sigma(1)}+ \nu_{\sigma(2)}, \tau, v_{23}=\nu_{\sigma(3)}+ \nu_{\sigma(4)}, -\tau$ can be considered as the oriented boundary of a parallelogram, yielding an orientation to the parallelogram. A necessary condition for the existence of an annular domain is that a parallel strip between the two sides $\pm \tau$ be included inside the periodgon. For this to occur, it is necessary that the parallelogram be negatively oriented. Hence a bifurcation creating an annular domain is one that reverses the orientation of the parallelogram. 
This occurs precisely when  the four vectors $v_{12}, \tau, v_{23}, -\tau$ become aligned: see Figure~\ref{fig:bif_periodic}. The corresponding bifurcation creating an annular domain is a heteroclinic loop with no return map in a rotated vector field. It appears in Figure~\ref{fig:bif_per_bis}.\end{proposition}

Note that the parallelogram is positively oriented in all the periodgons of Figure~\ref{deformed_periodgon} and that none of the corresponding vector fields has an annular domain.

\begin{figure}\begin{center}
\subfigure[An annular domain]{\includegraphics[width=3.8cm]{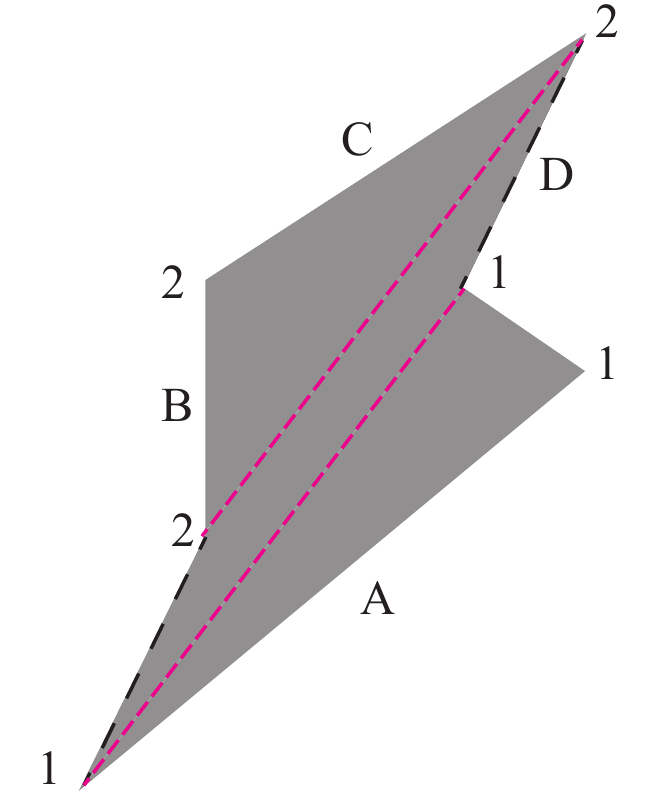}}\quad \subfigure[The bifurcation]{\includegraphics[width=3.8cm]{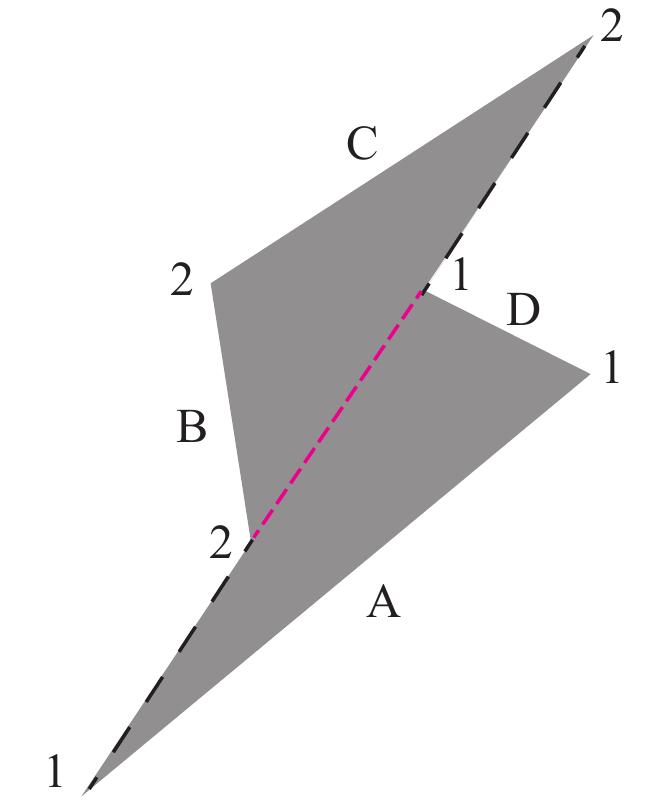}}\quad \subfigure[No annular domain]{\includegraphics[width=3.8cm]{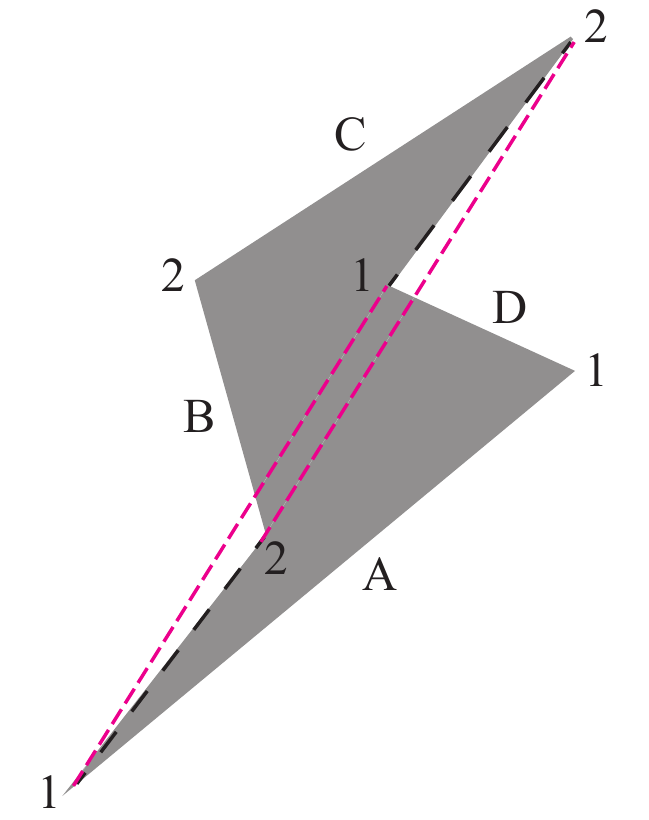}}\caption{The bifurcation of the periodgon in (b) between (a), where the vector field has an annular domain, and (c), where the vector field has no annular domain.}\label{fig:bif_periodic}\end{center}\end{figure}
\begin{figure}\begin{center}
\includegraphics[width=3.8cm]{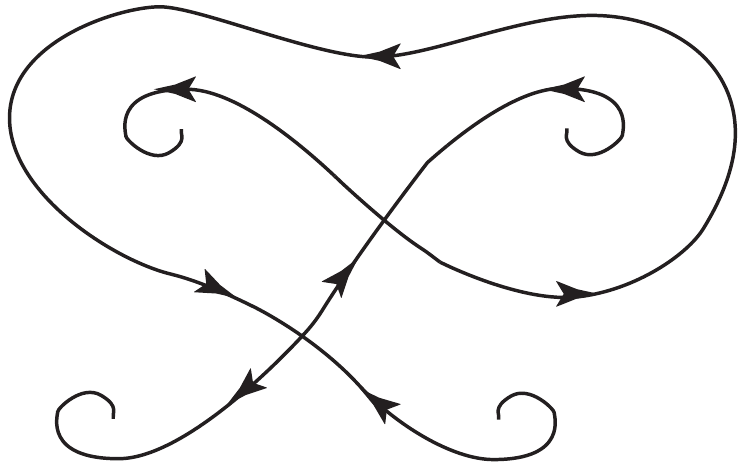}\caption{A heteroclinic loop with no return map. }\label{fig:bif_per_bis}\end{center}\end{figure}

The bifurcation described in Proposition~\ref{prop:ann_domain}  explicitly occurs inside rational vector fields of degree $4$. Indeed, consider the family 
\begin{equation}\dot z = i \frac{\left((z-a)^2+\frac12\right)\left((z^2-2)^2+\frac12\right)}{1-z^2},\label{family_heart}\end{equation} depending on real $a$. The bifurcation of heteroclinic loop occurs for $a\in(-0.8,-0.6)$, see Figure~\ref{fig:heart}. This special case was called the \emph{heart} by Ilyashenko in \cite{I} in planar vector fields. The bifurcation diagram in the generic case was studied by Dukov \cite{D}. The poles are transformed into saddles if one multiplies the vector field by $(1-z^2)(1-\bar z^2)$. One specific feature of the bifurcation diagram is the existence of two semi-infinite sequences of bifurcation curves where heteroclinic connections occur. For each $n\in \N$, there is one heteroclinic connection starting from the first (resp. second) pole to the second (resp. first pole) and making exactly $n$-turns and a fraction around two anti-saddles. The phenomenon has been called \emph{sparkling separatrices} by Ilyashenko \cite{I}.

\begin{figure}\begin{center}
\subfigure[$a=-0.8$]{\includegraphics[height=3.5cm]{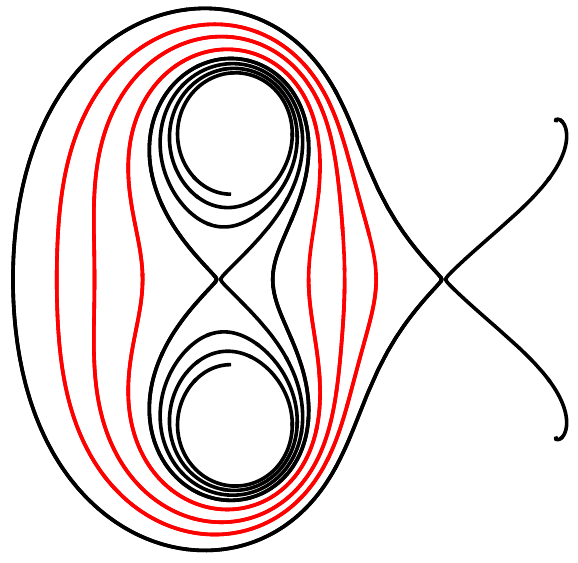}}\qquad\subfigure[$a\sim-0.6554$]{\includegraphics[height=3.5cm]{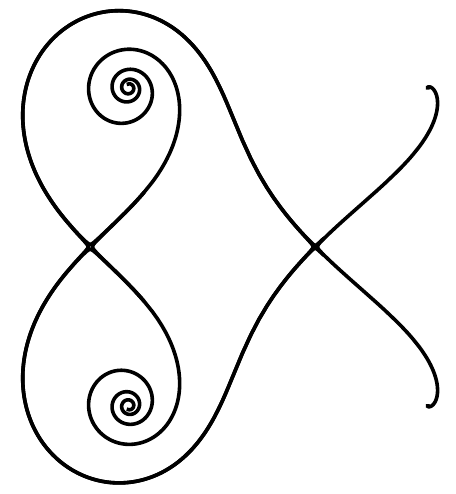}}\qquad\subfigure[$a=-0.6$]{\includegraphics[height=3.5cm]{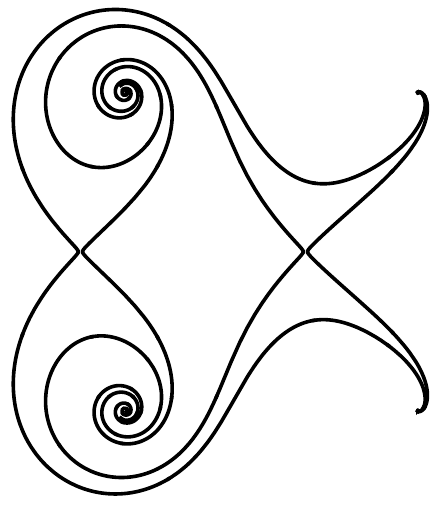}}\caption{The heart bifurcation in \eqref{family_heart}. For $a<-0.6554$ we see an annulus of periodic solutions bounded by two homoclinic connections.
}\label{fig:heart}\end{center}\end{figure}

A heteroclinic loop is a real codimension 2 bifurcation. Hence its bifurcation diagram requires two parameters. This bifurcation diagram indeed occurs in rational vector fields of degree $4$ inside the family 
$$\dot z = i e^{i\theta} \frac{\left((z-a)^2+\frac12\right)\left((z^2-2)^2+\frac12\right)}{1-z^2}$$
depending on the two parameters $a$ and $\theta$: see Figure~\ref{fig:bif_diag}. The bifurcation diagram can also be read from the periodgon: see Figure~\ref{fig:bif_diag2}. 

\begin{figure}\begin{center}
\includegraphics[width=13cm]{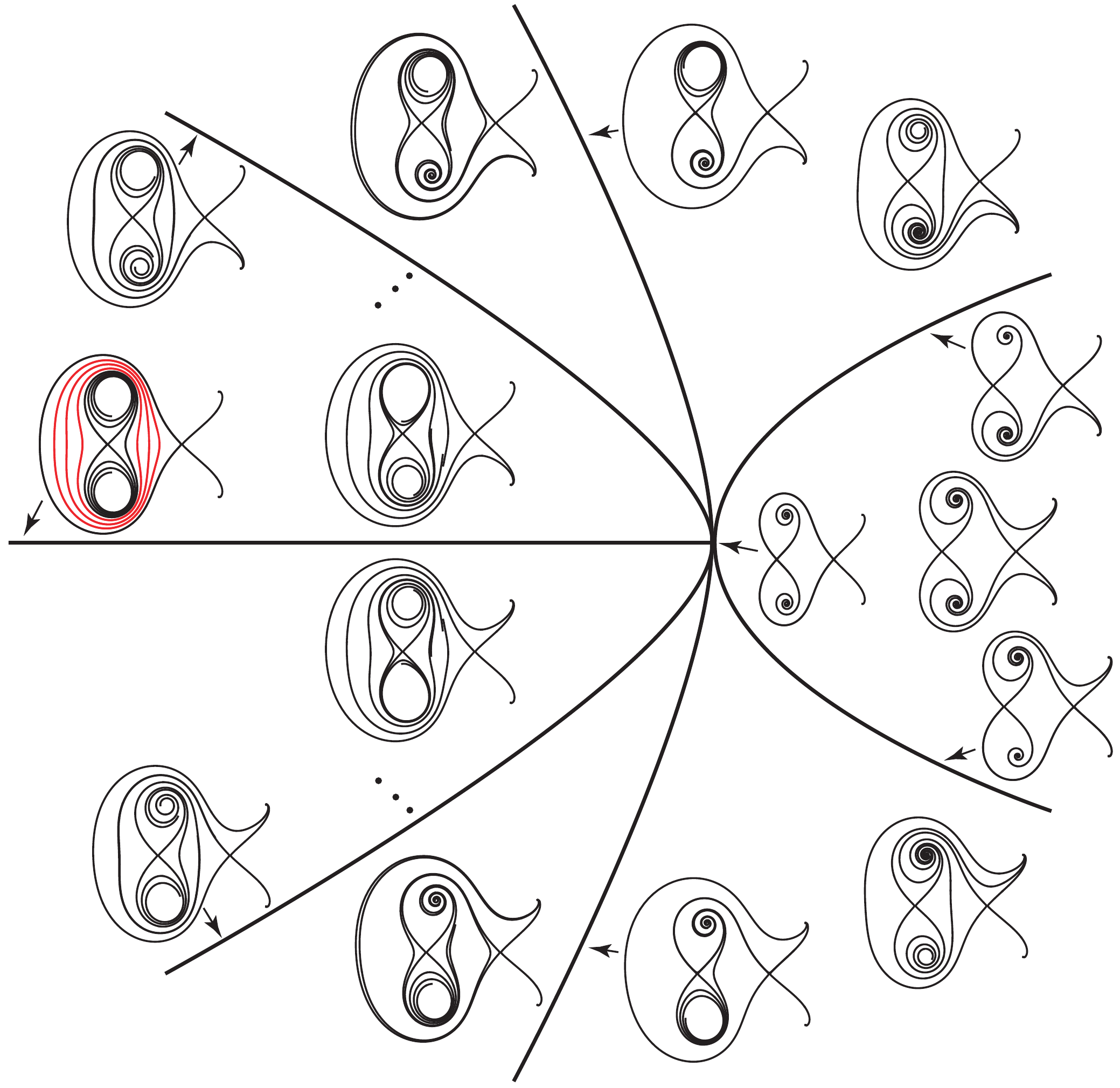}\caption{The topological $2$-parameter bifurcation diagram of the heteroclinic loop of Figure~\ref{fig:bif_per_bis}: an annulus of period orbits occurs along a half-curve in parameter space. There are two infinite sequences of heteroclinic connections occuring in the upper and lower left quadrants. }\label{fig:bif_diag}\end{center}\end{figure}

\begin{figure}\begin{center}
\subfigure[]{\includegraphics[width=5.6cm]{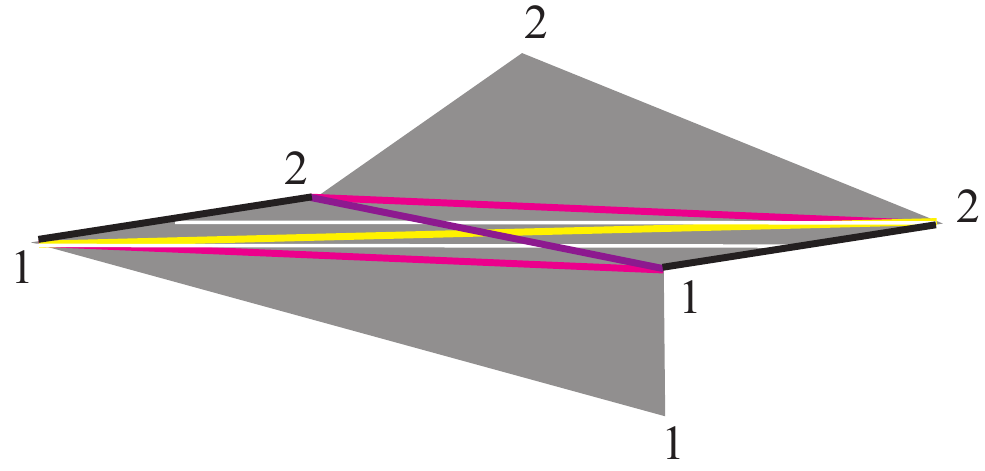}}\qquad\subfigure[]{\includegraphics[width=5.6cm]{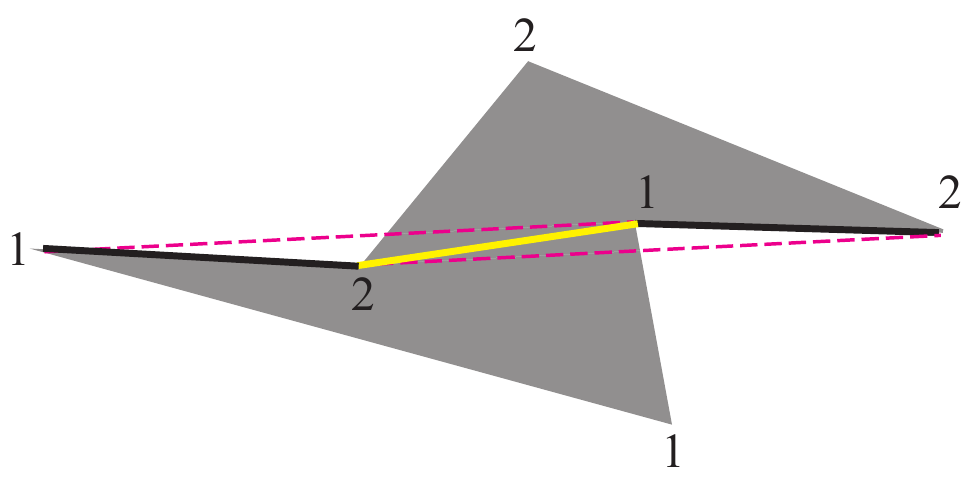}}\caption{The rotated periodgons of Figure~\ref{fig:bif_periodic}. 
Some  homoclinic or heteroclinic connections occur along the thick segments (and the white segments in (a)) when the periodgon is turned so that these are horizontal. These phenomena occur along the  bifurcation curves of Figure~\ref{fig:bif_diag}.
In (a) there is an infinite sequence of saddle connections inside the annular domain from 1 to 2 (when oriented form left to right) of increasing length and decreasing slope: the shortest is the black line, then the yellow line, then the white line, etc., and an infinite sequence of saddle connections from 2 to 1, the shortest of which is the purple line.
Only two of these saddle connections survive in (b) (represented by the black and yellow lines in the figure).
}\label{fig:bif_diag2}\end{center}\end{figure}

\begin{remark} The real codimension $2$ bifurcation of heteroclinic loop of Figure~\ref{fig:bif_per_bis} does not split the parameter space. But since we are only interested in the bifurcation of the \lq\lq shape\rq\rq\ of the periodgon, regardless of its rotations, 
 we quotient the parameter space by the group of rotations  \eqref{rotated_beta}, which in the family \eqref{vf_4} are induced by the action 
$$\begin{cases}z\mapsto e^{-i\frac{\beta}{3}}z,\\ 
(a,\eta_0,\eta_1,\eta_2,\eta_3)\mapsto(e^{i\frac{2\beta}{3}}a,e^{-i\frac{4\beta}{3}}\eta_0,e^{-i\beta}\eta_1,e^{-i\frac{2\beta}{3}}\eta_2,e^{-i\frac{\beta}{3}}\eta_3).\end{cases}$$
Then this becomes a real codimension $1$ bifurcation, which splits the quotient parameter space. \end{remark}

\subsection{Realization of a generalized periodgon and star domain} 
\begin{theorem} \label{thm:realization}
	Let $\nu_1, \nu_2,\nu_3,\nu_4\in \C^*$ such that $\nu_1+\nu_2+\nu_3+\nu_4=0$ and let $\tau\in\C^*$ be given. Let us consider the two (circular) sequences below.  \begin{enumerate}
\item $\nu_{\sigma(1)}, \nu_{\sigma(2)}, \tau, \nu_{\sigma(3)}, \nu_{\sigma(4)}, -\tau$ (for configuration $(2,2)$);
\item $\nu_{\sigma(1)}, \nu_{\sigma(2)},  \nu_{\sigma(3)}, \tau,\nu_{\sigma(4)}, -\tau$ (for configuration $(3,1)$).\end{enumerate} 
We consider each sequence as the ordered sequence of sides of a polygon over a translation surface with ramification points at the vertices. 
As a first condition we ask that the closed polygonal curve with sides given by the sequence is the oriented boundary (in the positive direction, i.e. the domain is to the left of the boundary) of an unbounded open domain in $\C$, the complement of which is a (possibly degenerate) polygon. 
In particular no sides can intersect transversally in their interior, but some sides can coincide with opposite direction.  
Let us now attach numbers $1$ or $2$ (representing the two poles) to the vertices of the polygon with the rule that the same number is attached to the two ends of a side $\nu_{\sigma(j)}$ and different numbers are attached to the two ends of  $\pm \tau$. 
We ask that the sum of the inner angles of the periodgon attached to each of the two numbers is equal to:   \begin{enumerate} 
\item $2\pi$ at each pole for configuration $(2,2)$;
\item $\pi$ (resp. $3\pi$) for the pole  attached to $3$ (resp. $1$) sides (resp. side) $\nu_j$ for configuration $(3,1)$.\end{enumerate}
Then the polygon can be realized as a generalized periodgon of a rational vector field of degree $4$. 

\end{theorem}\begin{proof} Let us  build the associated star domain by adding orthogonal semi-infinite strips of width $\nu_{j}$ on the corresponding sides of the polygon, which we  endow of the constant vector field $\dot t=1$. We glue together the two sides $\tau$ and $-\tau$, and the sides of the semi-infinite strips (yielding semi-infinite cylinders), thus obtaining a translation surface with conical singularities at the points $1$ and $2$. Removing these two  points, the translation surface is conformally equivalent to $\CP^1$ minus six points. Because the total angle is $4\pi$ at each of the points $1$ and $2$, a local uniformizing coordinate is given by $Z= \sqrt{t-t_0}$ with transforms $\dot t=1$ into $\dot Z= \frac1{2Z}$, i.e. the point is a pole. A uniformizing coordinate at the end of the semi-infinite cylinders of width $\nu_j$ is $Z = \exp\left(\frac{ 2\pi it}{\nu_j}\right)$, transforming $\dot t=1$ into $\dot Z= \frac{2\pi i}{\nu_j}Z$, i.e. $z_j$ is a simple singular point. Hence we have constructed a rational vector field of degree $4$ on $\CP^1$.
\end{proof}

\begin{remark} Note that all conditions enumerated in the theorem are necessary conditions. Hence the theorem is a characterization of realizable generalized periodgons.
The realisation map is however not bijective: one needs to identify those generalized periodgons that give rise to the same translation model and therefore the same vector field (up to a Moebius transformation). 	
An open question is to find necessary and sufficient condition for realizable periodgons (not generalized).\end{remark}

\subsection{Polynomial vector fields of degree 4} 

As a consequence we can see what kind of periodgons can appear for a polynomial system of degree $4$. 

The periodgon of a polynomial system of degree $4$ is always a planar polygon (its projection from the translation surface of $t(z)$ to $\C$ has no self-intersection). It can be flat (four sides aligned) with three sides in one direction and one in the other as in $\dot z = i z(z^3-1)$, or with two sides in each direction as in $\dot z = (z^2+1)(z^2+2)$. 

The periodgon bifurcates generically when two adjacent sides become aligned in opposite directions, as in $\dot z = (z^2-1) (z-i)(z-2i) $. In that case, one of the periodic zone is surrounded by two homoclinic loops (see Figures~\ref{two_concentric_homo} and \ref{bifur_periodgon}). 
Indeed, we can remark that as soon as one center is surrounded by two  homoclinic loops, it separates the other singular points in two groups of, respectively, one and two points. The singular point in the group of $1$ is then surrounded by a homoclinic loop and hence, necessarily a center. 

\begin{figure}\begin{center}
		\includegraphics[width=5cm]{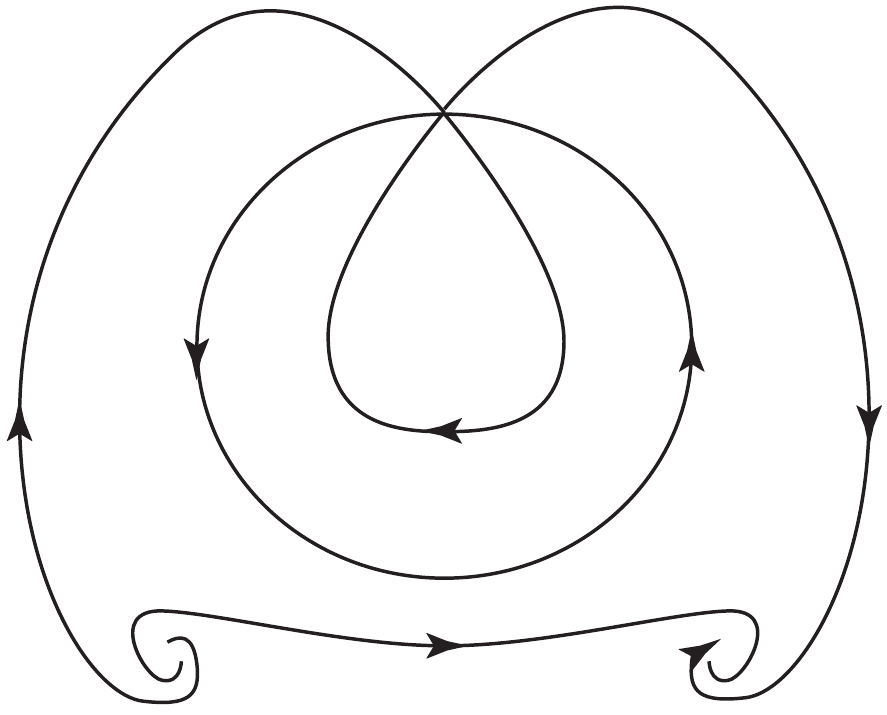}\caption{The phase portrait of $\dot z = (z^2-1) (z-i)(z-2i) $ with two homoclinic loops surrounding $z=2i$.}\label{two_concentric_homo}\end{center}\end{figure}

\begin{figure}\begin{center}
		\subfigure[]{\includegraphics[width=5cm]{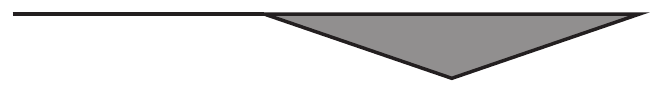}}\qquad\subfigure[]{\includegraphics[width=5cm]{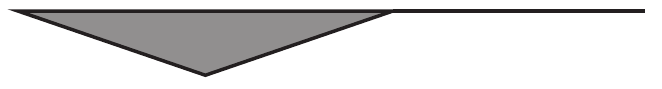}}\\
		\subfigure[]{\includegraphics[width=5cm]{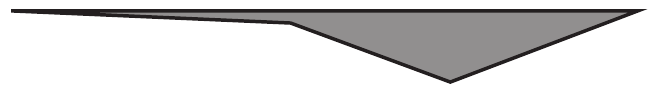}}\qquad\subfigure[]{\includegraphics[width=5cm]{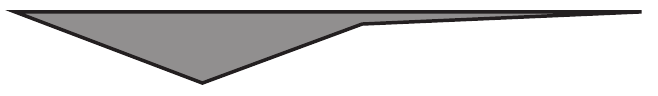}}
		\caption{In (a) and (b) the two periodgons of $\dot z = (z^2-1) (z-i)(z-2i) $, and in (c) and (d) the unique periodgons on both sides of the bifurcation. }\label{bifur_periodgon}\end{center}\end{figure}

\section*{Acknowledgements} The authors are grateful to Arnaud Ch\'eritat for stimulating discussions.

\end{document}